\documentclass{elsarticle}

\usepackage{amsmath, amstext, amsgen, amsbsy, amsopn, amsfonts, amssymb, graphicx,psfrag}
\usepackage{amsthm}
\usepackage{comment}
\usepackage[normalem]{ulem}
\usepackage{pdfsync}
\usepackage{epstopdf}
\usepackage{color}
\usepackage[thinlines]{easytable}
\usepackage{multirow}
\usepackage{multicol}
\usepackage{hyperref}

\newtheorem{theorem}{Theorem}[section]
\newtheorem{lemma}[theorem]{Lemma}
\theoremstyle{definition}
\newtheorem*{definition}{Definition}

\theoremstyle{remark}
\newtheorem{remark}{Remark}[section]

\begin{document}

\begin{frontmatter}


\title{$N$-quandles of links}

\author{Blake Mellor and Riley Smith\fnref{riley}}
\address{Loyola Marymount University, 1 LMU Drive, Los Angeles, CA 90045}
\fntext[riley]{This paper is based in part on the second author's senior thesis project at LMU.}

\date{}
\begin{abstract} 
The fundamental quandle is a powerful invariant of knots and links, but it is difficult to describe in detail. It is often useful to look at quotients of the quandle, especially finite quotients.  One natural quotient introduced by Joyce \cite{JO} is the $n$-quandle.  Hoste and Shanahan \cite{HS2} gave a complete list of the knots and links which have finite $n$-quandles for some $n$.  We introduce a generalization of $n$-quandles, denoted $N$-quandles (for a quandle with $k$ algebraic components, $N$ is a $k$-tuple of positive integers). We conjecture a classification of the links with finite $N$-quandles for some $N$, and we prove one direction of the classification.
\end{abstract}


\end{frontmatter}

\section{Introduction}
Every oriented knot and link $L$ has a fundamental quandle $Q(L)$. For tame knots, Joyce \cite{JO, JO2} and Matveev \cite{Ma} showed that the fundamental quandle is a complete invariant (up to a change of orientation).  Unfortunately, classifying quandles is no easier than classifying knots - in particular, except for the unknot and Hopf link, the fundamental quandle is always infinite.  This motivates us to look at quotients of the quandle, and particularly finite quotients, as possibly more tractable ways to distinguish quandles. Since knot and link quandles are residually finite \cite{BSS1, BSS2}, any pair of elements are distinguished by some finite quotient, so there are many potentially useful quotients to study.  One such quotient that has a topological interpretation is the $n$-quandle introduced by Joyce \cite{JO, JO2}.  Hoste and Shanahan \cite{HS2} proved that the $n$-quandle $Q_n(L)$ is finite if and only if $L$ is the singular locus (with each component labeled $n$) of a spherical 3-orbifold with underlying space $\mathbb{S}^3$. This result, together with Dunbar's \cite{DU} classification of all geometric, non-hyperbolic 3-orbifolds, allowed them to give a complete list of all  knots and links in $\mathbb{S}^3$ with finite $n$-quandles for some $n$ \cite{HS2}; see Table~\ref{linktable}.  Many of these finite $n$-quandles have been described in detail \cite{CHMS, HS1, Me}.

However, Dunbar's classification also includes orbifolds whose singular locus is a link with different labels on different components. This motivates us to generalize the idea of an $n$-quandle to define an $N$-quandle, where $N = (n_1,\dots, n_k)$ (and $k$ is the number of components of the link).  We conjecture
\medskip

\noindent {\bf Main Conjecture.} {\em A link $L$ with $k$ components has a finite $(n_1,\dots, n_k)$-quandle if and only if there is a spherical orbifold with underlying space $\mathbb{S}^3$ whose singular locus is the link $L$, with component $i$ labeled $n_i$.}
\medskip

In this paper we will define the $N$-quandle of a link and prove it is an invariant of isotopy.  We will then prove half of our conjecture, and show that if there is a spherical orbifold with underlying space $\mathbb{S}^3$ whose singular locus is the link $L$, with component $i$ labeled $n_i$, then the $(n_1,\dots, n_k)$-quandle of $L$ is finite.

\begin{table}[htbp]
{$
\begin{array}{cccc}
\includegraphics[width=1.25in,trim=0 0 0 0,clip]{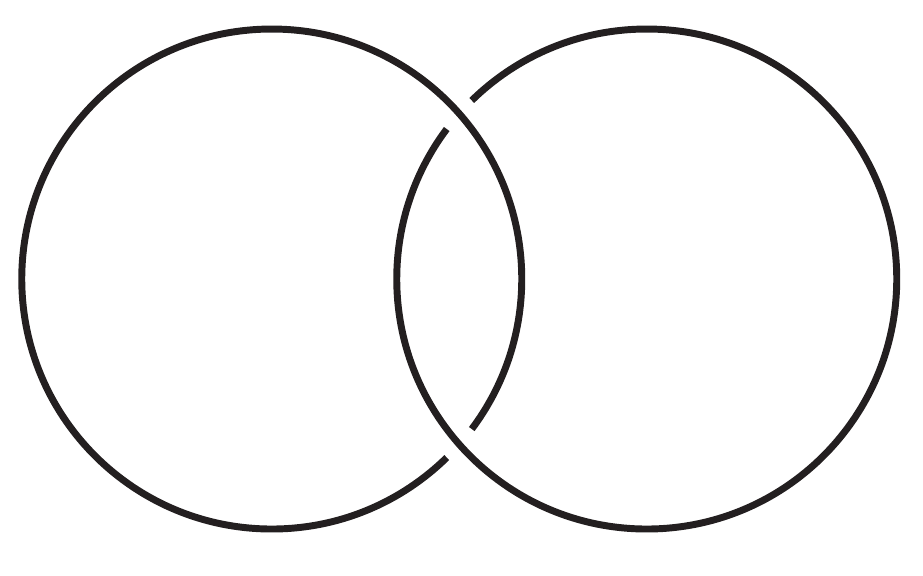}   & \includegraphics[width=1.25in,trim=0 0 0 0,clip]{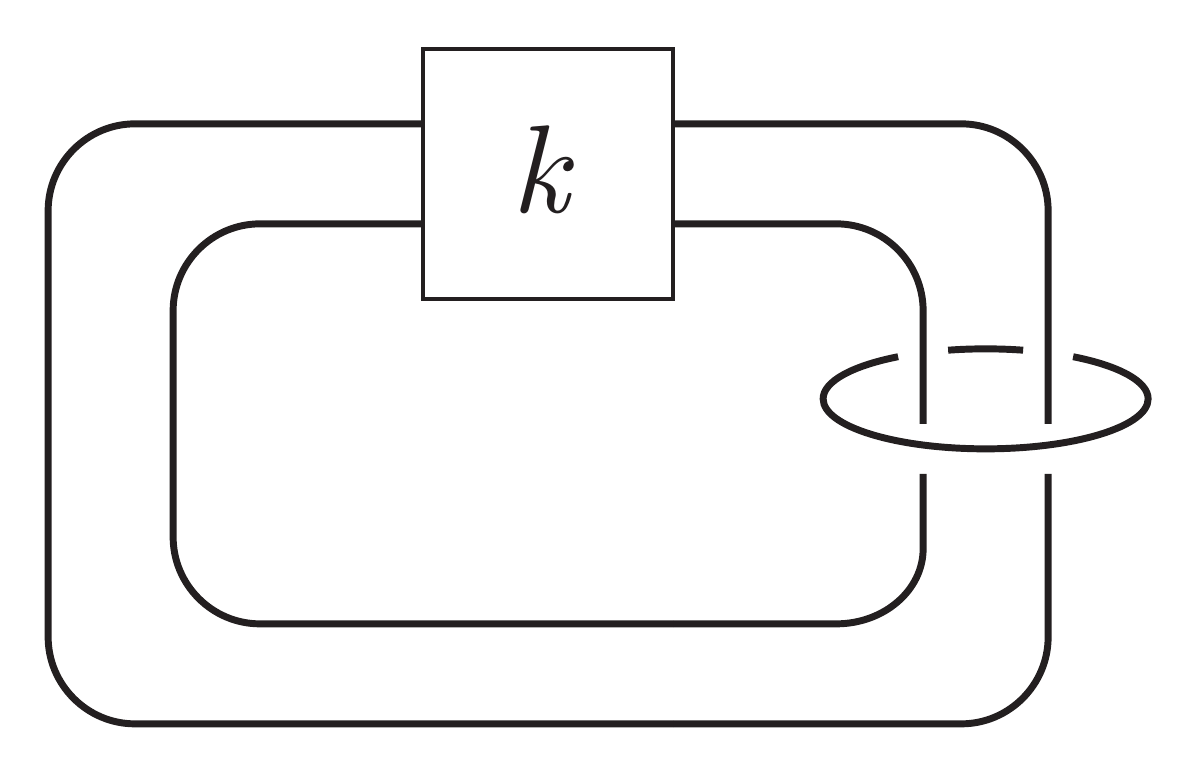}  & \includegraphics[width=1.0in,trim=0 0 0 0,clip]{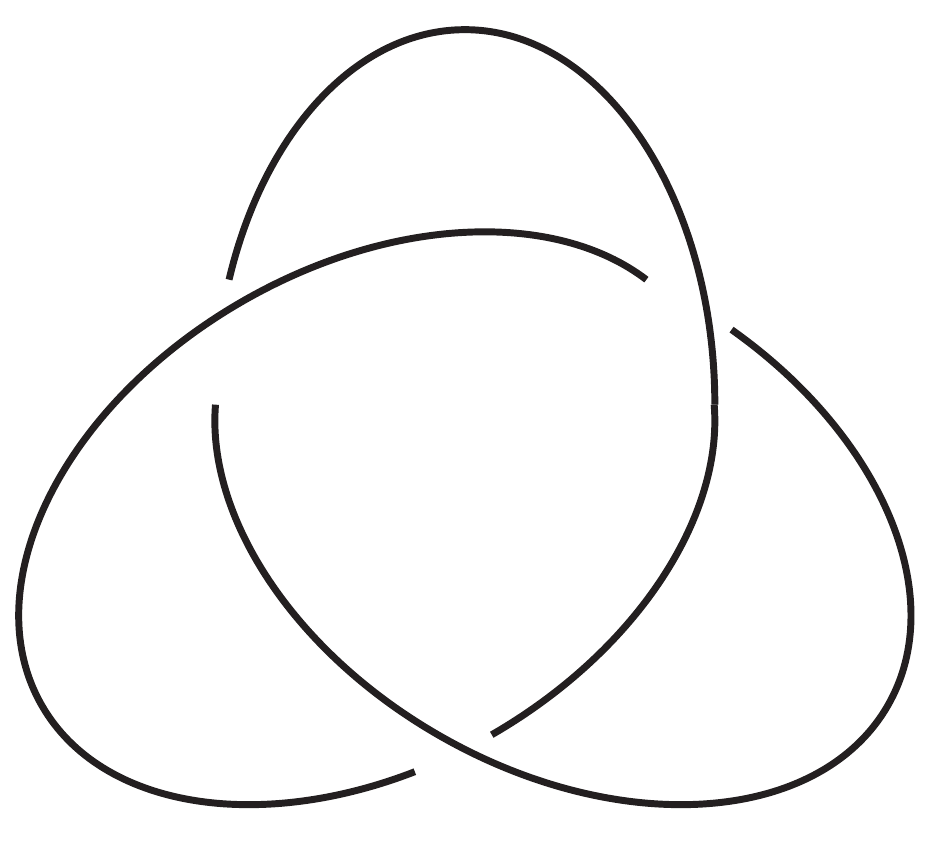} \\
\scriptstyle n > 1 & \scriptstyle k \neq 0,\  n=2& \scriptstyle n=3, 4, 5 \\
\\
\includegraphics[width=1.0in,trim=0 0 0 0,clip]{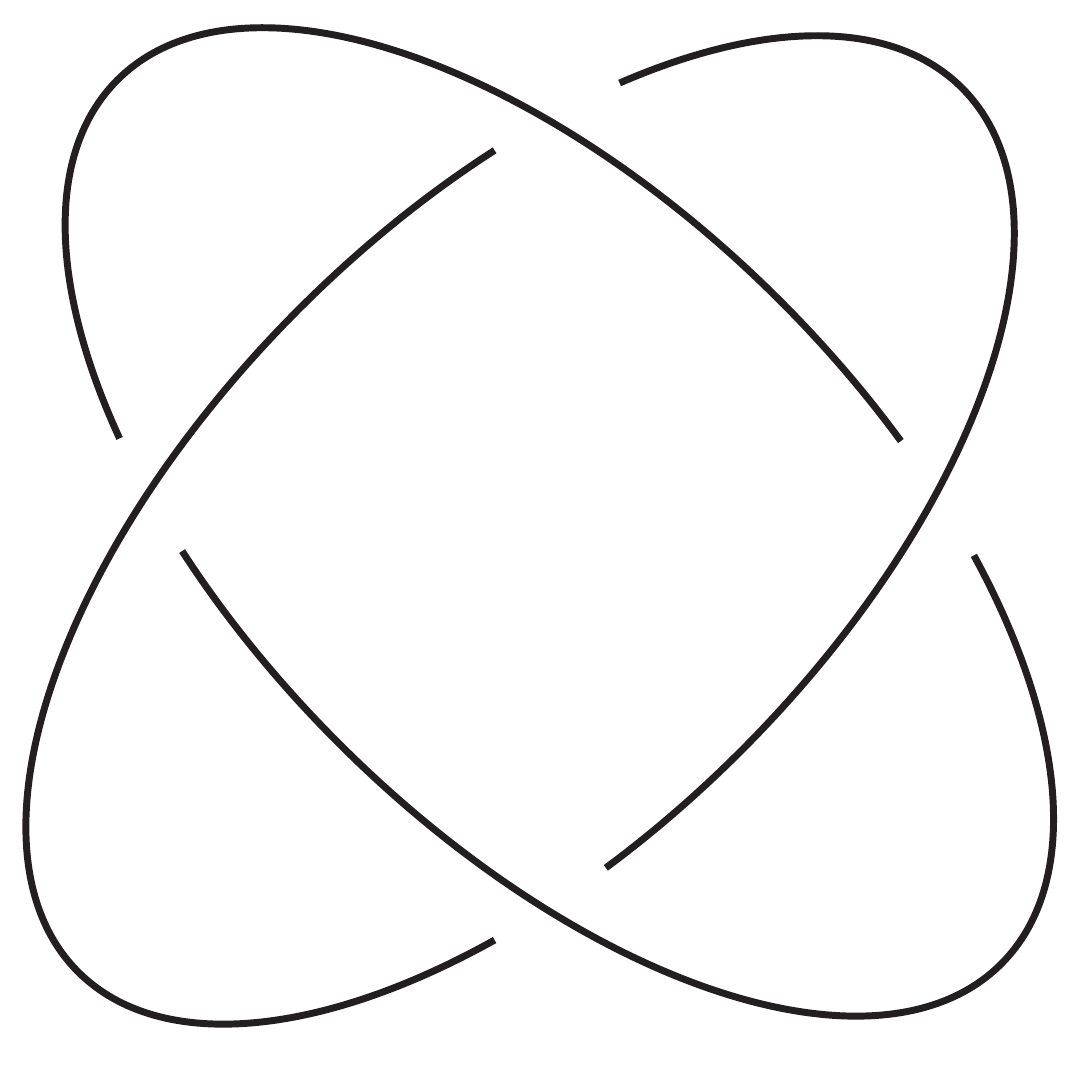}  & \includegraphics[width=1.0in,trim=0 0 0 0,clip]{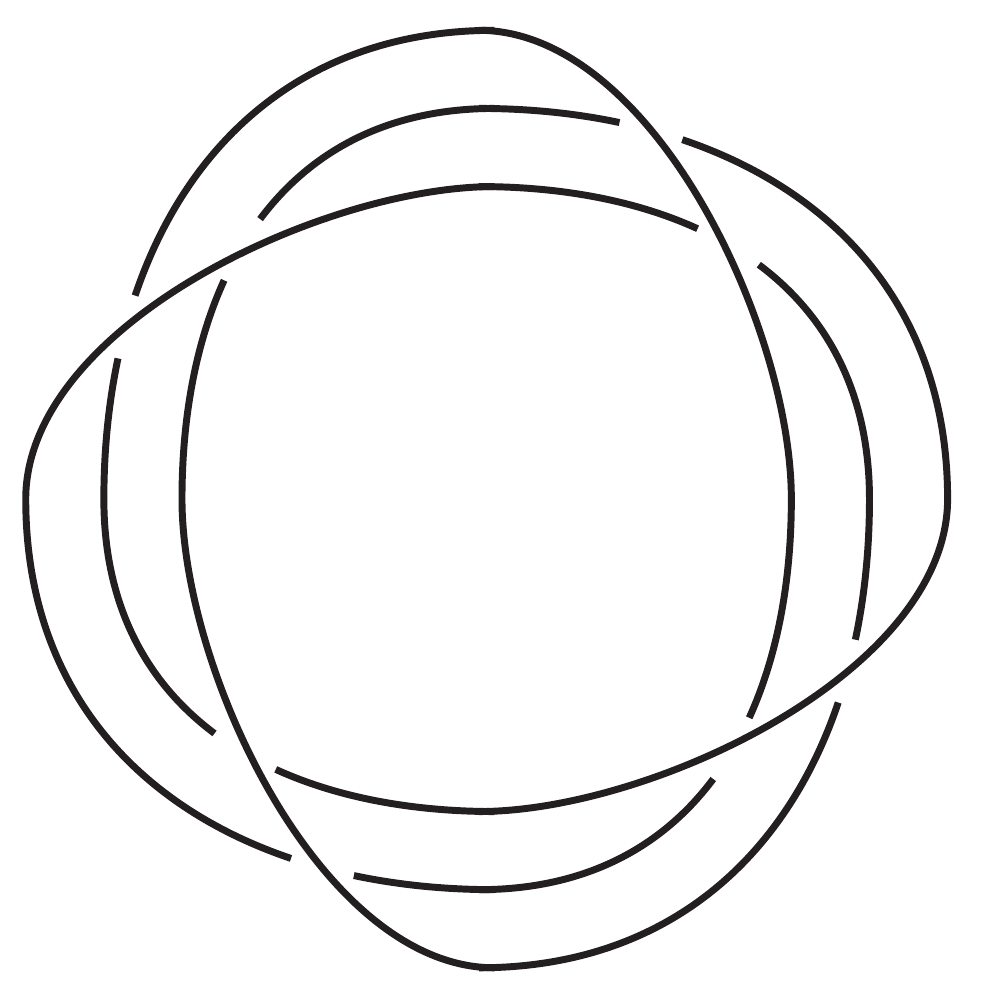}  & \includegraphics[width=1.0in,trim=0 0 0 0,clip]{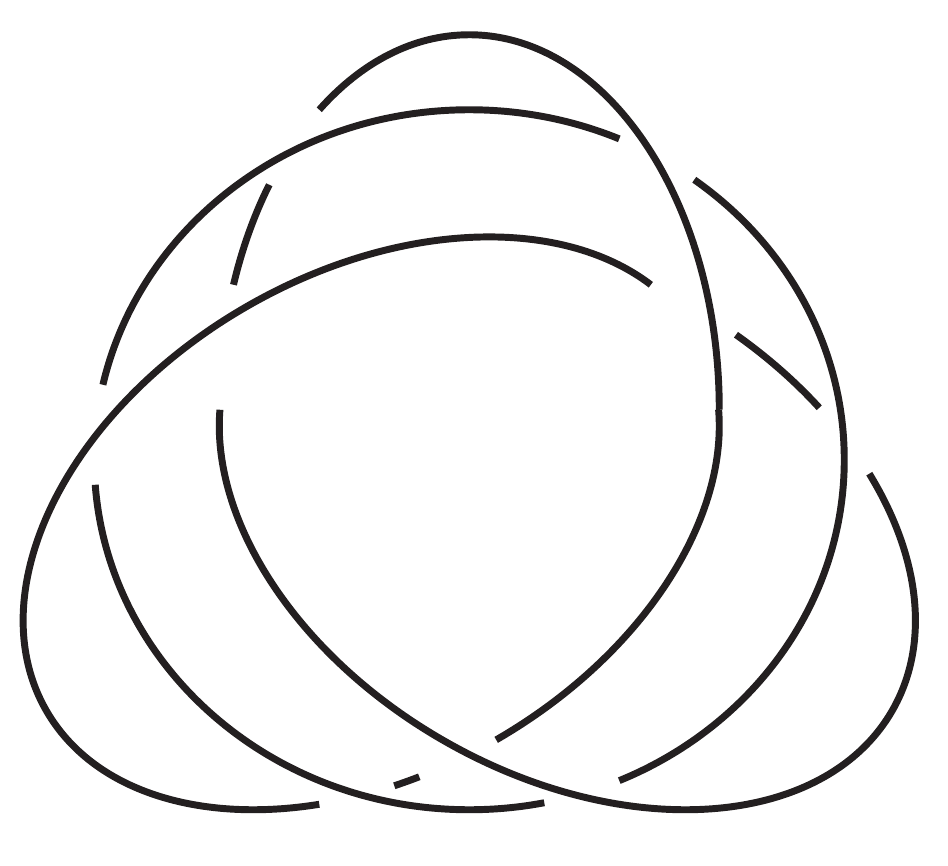} \\
\scriptstyle n =3 & \scriptstyle n=2& \scriptstyle n=2 \\
\\
\includegraphics[width=1.0in,trim=0 0 0 0,clip]{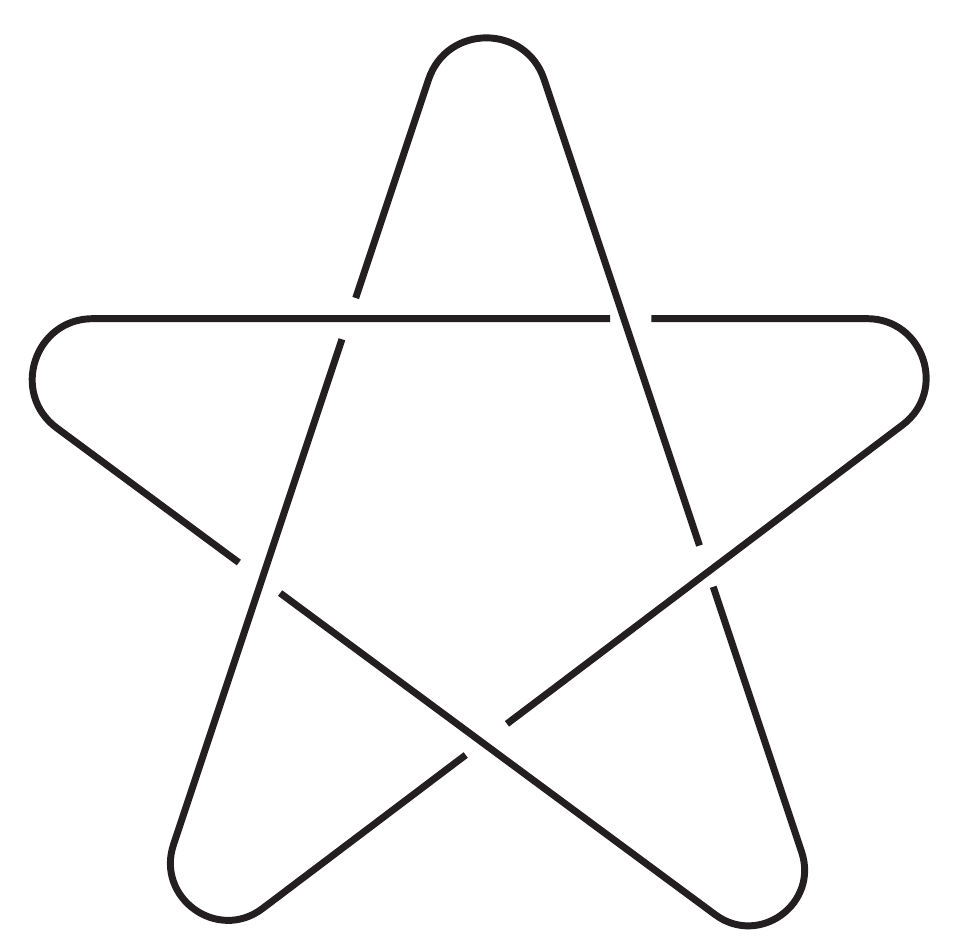}  & \includegraphics[width=1.0in,trim=0 0 0 0,clip]{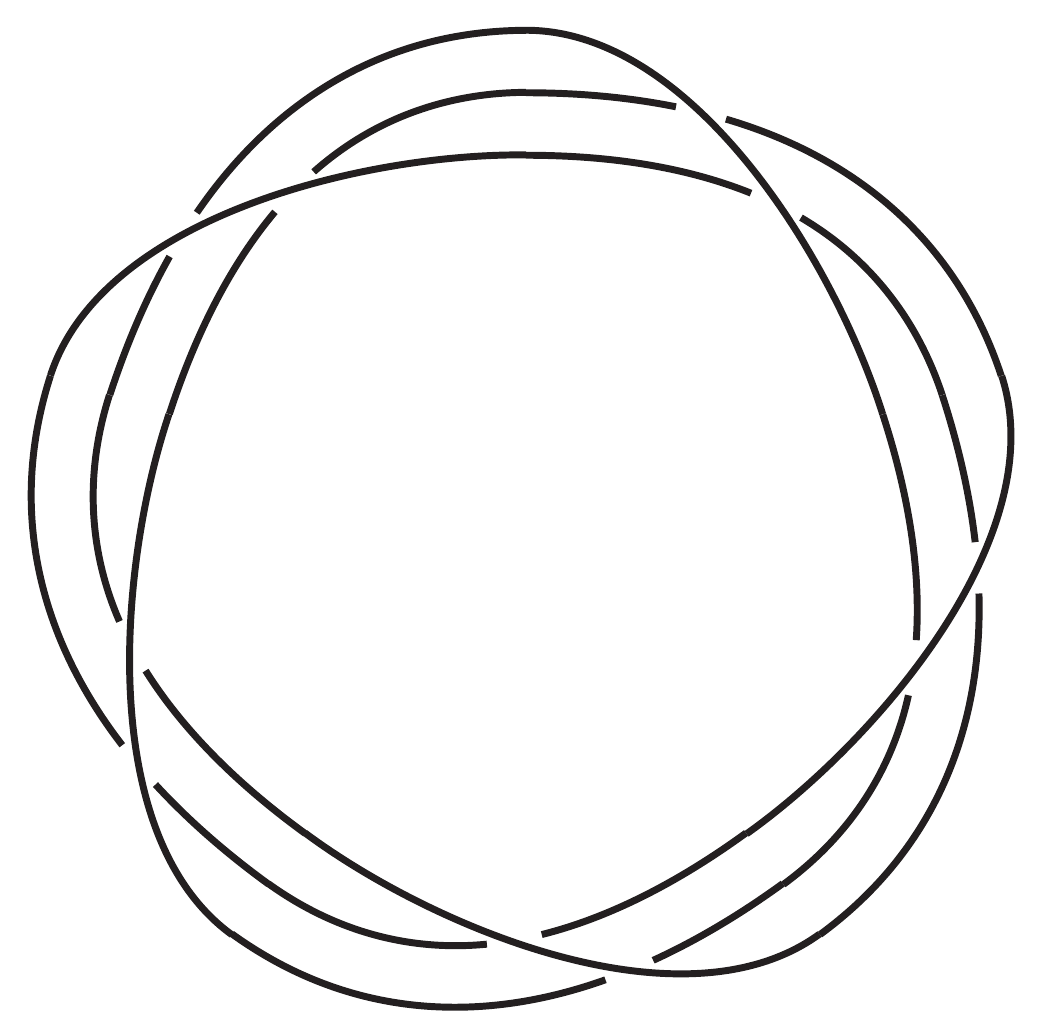}  & \includegraphics[width=1.25in,trim=0 0 0 0,clip]{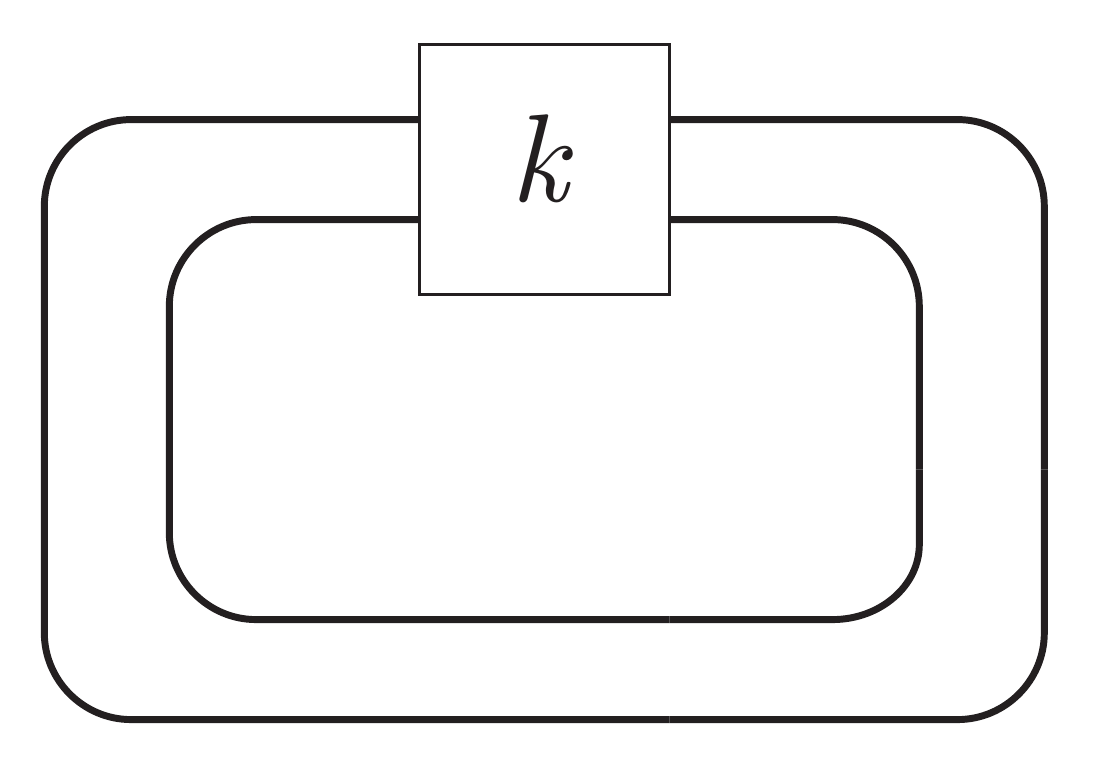} \\
\scriptstyle n =3 & \scriptstyle n=2& \scriptstyle k\neq0,\ n=2 \\
\\
\includegraphics[width=1.25in,trim=0 0 0 0,clip]{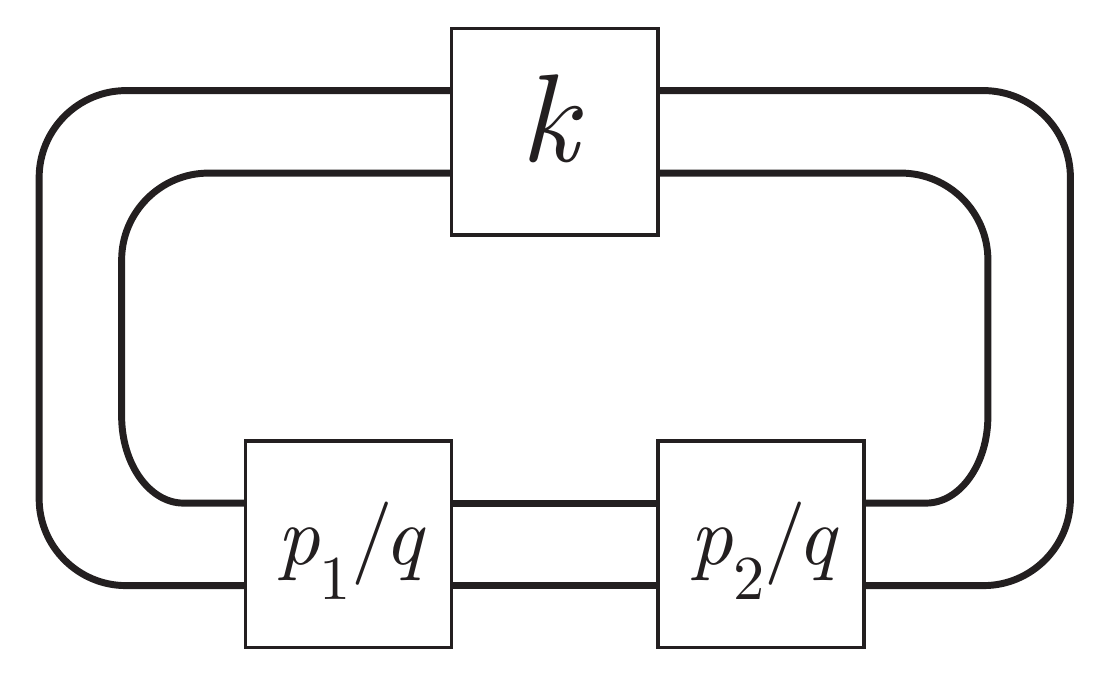}  & \includegraphics[width=1.15in,trim=0 5pt 0 0,clip]{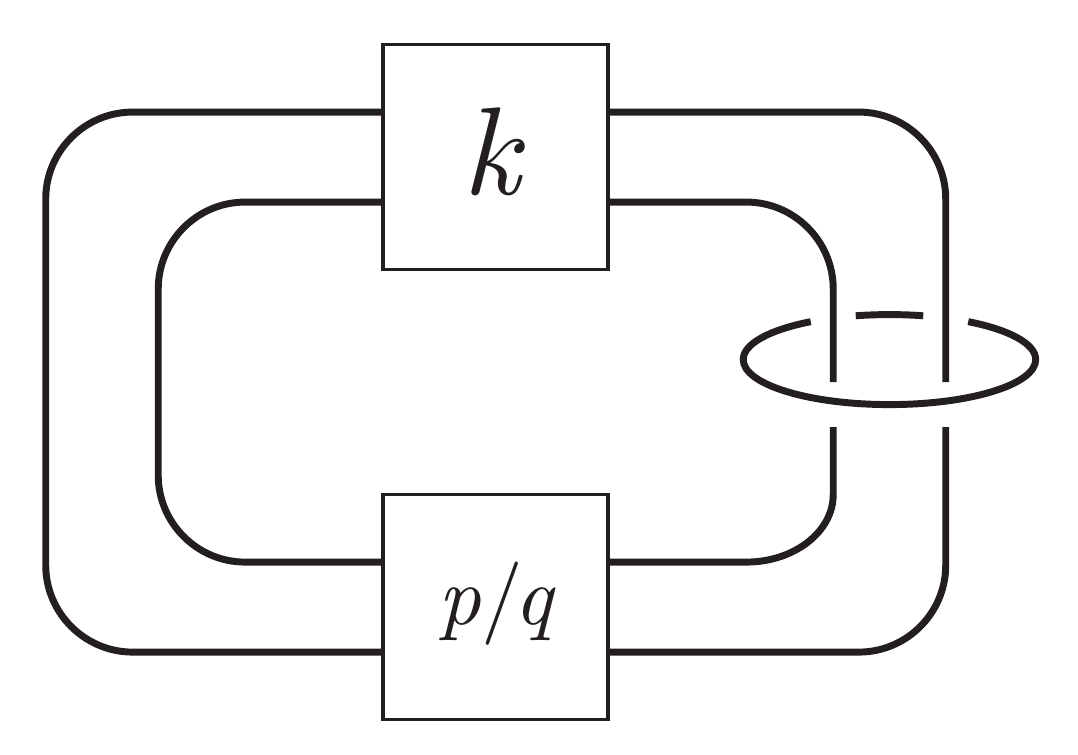} &  \includegraphics[width=1.65in,trim=0 0 0 0,clip]{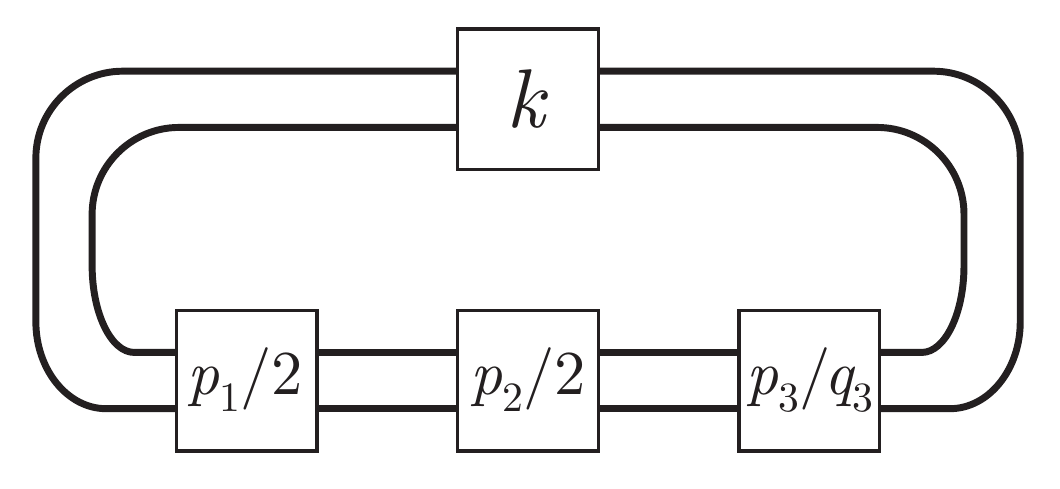} \\
\scriptstyle k+p_1/q+p_2/q \neq 0,\ n =2  &\scriptstyle n=2& \scriptstyle k+p_1/2+p_2/2+p_3/q_3 \neq 0,\ n =2 \\
\\
\includegraphics[width=1.65in,trim=0pt 0pt 0pt 0pt,clip]{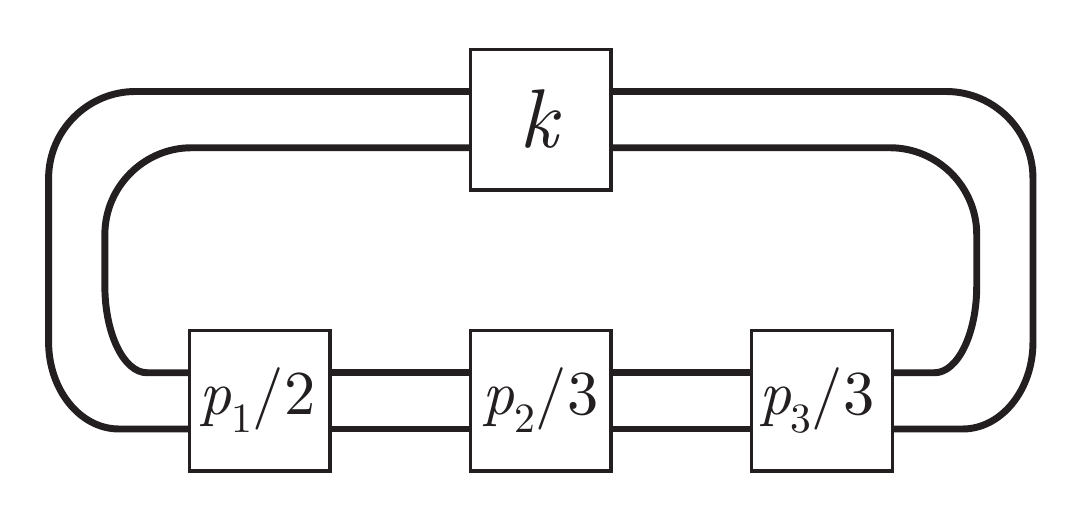}  & \includegraphics[width=1.65in,trim=0pt 0pt 0pt 0pt,clip]{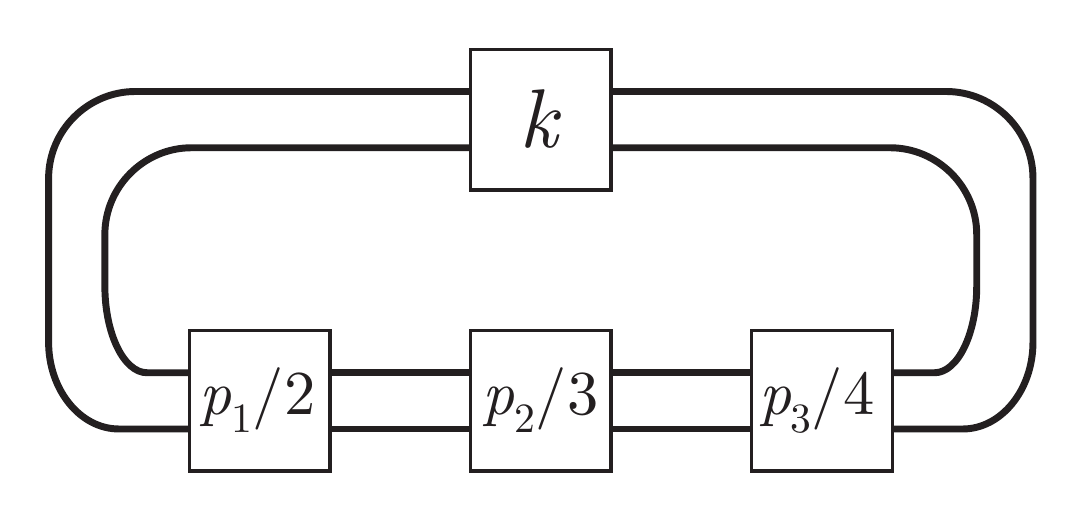}  & \includegraphics[width=1.65in,trim=0pt 0pt 0pt 0pt,clip]{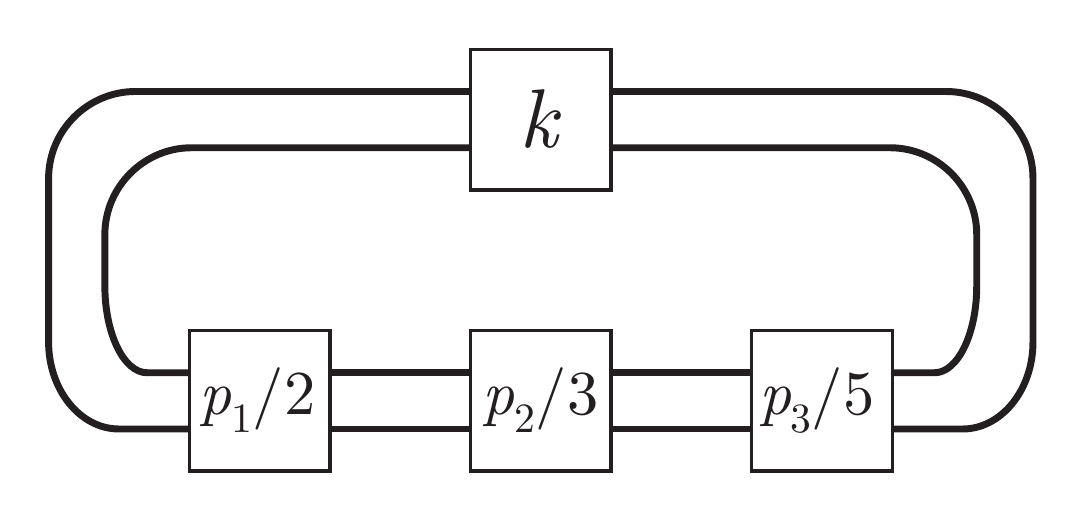}\\
\scriptstyle k+p_1/2+p_2/3+p_3/3 \neq 0,\ n =2  & \scriptstyle k+p_1/2+p_2/3+p_3/4 \neq 0,\ n =2  & \scriptstyle k+p_1/2+p_2/3+p_3/5 \neq 0,\ n =2
\end{array}
$}
\caption{\parbox{3in}{Links $L \in \mathbb{S}^3$ with finite $Q_n(L)$. \newline Here \fbox{$k$} represents $k$ right-handed half-twists, and \fbox{$p/q$} represents a rational tangle.}}
\label{linktable}
\end{table}

\section{Quandles, $n$-quandles and $N$-quandles} \label{S:quandles}

We begin with a review of the definition of a quandle and its associated $n$-quandles. We refer the reader to \cite{JO}, \cite{JO2}, \cite{FR}, and \cite{WI} for more detailed information.

A {\it  quandle} is a set $Q$ equipped with two binary operations $\rhd$ and $\rhd^{-1}$ that satisfy the following three axioms:
\begin{itemize}
\item[\bf A1.] $x \rhd x =x$ for all $x \in Q$.
\item[\bf A2.] $(x \rhd y) \rhd^{-1} y = x = (x \rhd^{-1} y) \rhd y$ for all $x, y \in Q$.
\item[\bf A3.] $(x \rhd y) \rhd z = (x \rhd z) \rhd (y \rhd z)$ for all $x,y,z \in Q$.
\end{itemize}

Each element $x\in Q$ defines a map $S_x:Q \to Q$ by $S_x(y)=y \rhd x$. The axiom A2 implies that each $S_x$ is a bijection and the axiom A3 implies that each $S_x$ is a quandle homomorphism, and therefore an automorphism. We call $S_x$ the {\it point symmetry at $x$}. The {\em inner automorphism group} of $Q$, Inn$(Q)$, is the group of automorphisms generated by the point symmetries.

It is important to note that the operation $\rhd$ is, in general, not associative. To clarify the distinction between $(x \rhd y) \rhd z$ and $x \rhd (y \rhd z)$, we adopt the exponential notation introduced by Fenn and Rourke in \cite{FR} and denote $x \rhd y$ as $x^y$ and $x \rhd^{-1} y$ as $x^{\bar y}$. With this notation, $x^{yz}$ will be taken to mean $(x^y)^z=(x \rhd y)\rhd z$ whereas $x^{y^z}$ will mean $x\rhd (y \rhd z)$. 

The following useful lemma from \cite{FR} describes how to re-associate a product in a quandle given by a presentation. 

\begin{lemma} \label{leftassoc}
If $a^u$ and $b^v$ are elements of a quandle, then
$$\left(a^u \right)^{\left(b^v \right)}=a^{u \bar v b v} \ \ \ \ \mbox{and}\ \ \ \ \left(a^u \right)^{\overline{\left(b^v \right)}}=a^{u \bar v \bar b v}.$$
\end{lemma}

Using Lemma~\ref{leftassoc}, elements in a quandle given by a presentation $\langle S \mid R \rangle$ can be represented as equivalence classes of expressions of the form $a^w$ where $a$ is a generator in $S$ and $w$ is a word in the free group on $S$ (with $\bar x$ representing the inverse of $x$).

If $n$ is a natural number, a quandle $Q$ is an {\em $n$-quandle} if $x^{y^n} =x$ for all $x,y \in Q$, where by $y^n$ we mean $y$ repeated $n$ times. Given a presentation $\langle S \,|\, R\rangle$ of $Q$, a presentation of  the quotient $n$-quandle $Q_n$ is obtained by adding the relations $x^{y^n}=x$ for every pair of distinct generators $x$ and $y$. 

The action of the inner automorphism group Inn$(Q)$ on the quandle $Q$ decomposes the quandle into disjoint orbits. These orbits are the {\em components} (or {\em algebraic components}) of the quandle $Q$; a quandle is {\em connected} if it has only one component. We generalize the notion of an $n$-quandle by picking a different $n$ for each component of the quandle.

\begin{definition} \label{D:Nquandle}
Given a quandle $Q$ with $k$ ordered components, labeled from 1 to $k$, and a $k$-tuple of natural numbers $N = (n_1, \dots, n_k)$, we say $Q$ is an {\em $N$-quandle} if $x^{y^{n_i}} = x$ whenever $x \in Q$ and $y$ is in the $i$th component of $Q$. 
\end{definition}

Note that the ordering of the components in an $N$-quandle is very important; the relations depend intrinsically on knowing which component is associated with which number $n_i$. 

Given a presentation $\langle S \,|\, R\rangle$ of $Q$, a presentation of  the quotient $N$-quandle $Q_N$ is obtained by adding the relations $x^{y^{n_i}}=x$ for every pair of distinct generators $x$ and $y$, where $y$ is in the $i$th component of $Q$. An $n$-quandle is then the special case of an $N$-quandle where $n_i = n$ for every $i$. 

\begin{remark}
The families of $n$-quandles and $N$-quandles are similar to other families of quandles that have been studied in a purely algebraic context.  These include reductive quandles, locally reductive quandles and (more generally) quandles with orbit series conditions \cite{BCNW, PR, BS, JPZ}.  For example, a quandle is {\em $n$-locally reductive} if $y^{x^n} = x$ for every pair of elements $x$ and $y$. In the future, it would be interesting to study the results of adding these conditions to link quandles.
\end{remark}

\section{Link quandles}

If $L$ is an oriented knot or link in $\mathbb{S}^3$, then a presentation of its fundamental quandle, $Q(L)$, can be derived from a regular diagram $D$ of $L$ by a process similar to the Wirtinger algorithm \cite{JO}. We assign a quandle generator $x_1, x_2, \dots , x_n$ to each arc of $D$, then at each crossing introduce the relation $x_i=x_k^{ x_j}$ as shown in Figure~\ref{crossing}. It is easy to check that the three Reidemeister moves do not change the quandle given by this presentation so that the quandle is indeed an invariant of the oriented link.

\begin{figure}[h]
$$\includegraphics[height=1in]{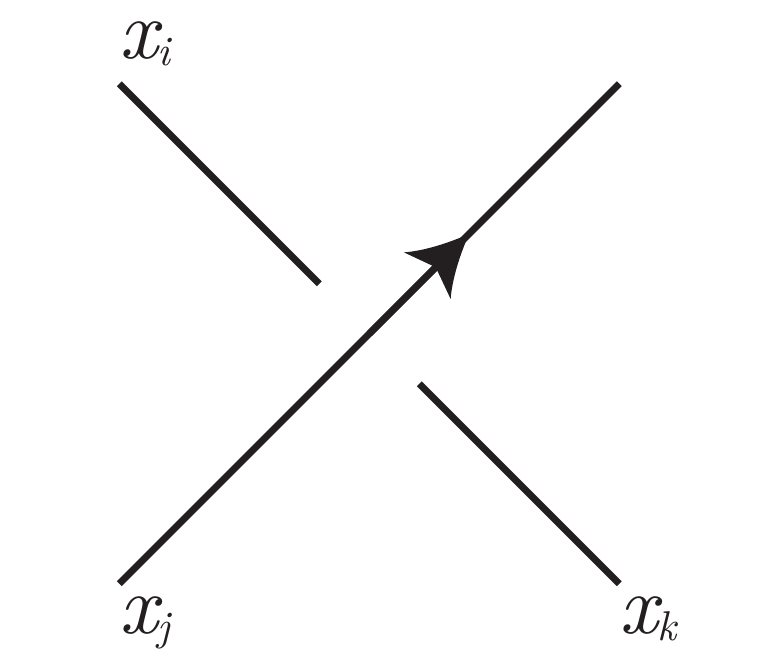}$$
\caption{The relation $x_i=x_k^{x_j}$ at a crossing.}
\label{crossing}
\end{figure}

If $n$ is a natural number, we can take the quotient $Q_n(L)$ of the fundamental quandle $Q(L)$ to obtain the fundamental $n$-quandle of the link. Hoste and Shanahan \cite{HS2} classified all pairs $(L,n)$ for which $Q_n(L)$ is finite (see Table \ref{linktable}).

Fenn and Rourke \cite{FR} observed that the components of the quandle $Q(L)$ are in bijective correspondence with the components of the link $L$, with each component of the quandle containing the generators of the Wirtinger presentation associated to the corresponding link component. So if we have a link $L$ of $k$ components, and label each component $c_i$ with a natural number $n_i$, we can let $N = (n_1, \dots, n_k)$ and take the quotient $Q_N(L)$ of the fundamental quandle $Q(L)$ to obtain the fundamental $N$-quandle of the link (this depends on the ordering of the link components). If $Q(L)$ has the Wirtinger presentation from a diagram $D$, then we obtain a presentation for $Q_N(L)$ by adding relations $x^{y^{n_i}} = x$ for each pair of distinct generators $x$ and $y$ where $y$ corresponds to an arc of component $c_i$ in the diagram $D$.

\begin{remark} \label{R:Nquandle}
It is worth observing that if $x_i^y = x_i$ for every {\em generator} $x_i$ of a quandle, then $x^y = x$ for every {\em element} $x$ of the quandle.  Say $x = x_1^{x_2x_3\dots x_m}$, where each $x_i$ is a generator.  Then
$$x^y = x_1^{x_2x_3\cdots x_m y} = x_1^{y(\bar{y}x_2 y)(\bar{y} x_3 y) \cdots (\bar{y} x_m y)} = (x_1^y)^{(x_2^y)(x_3^y) \cdots (x_m^y)} = x_1^{x_2x_3\dots x_m} = x.$$
We will use this fact when constructing Cayley graphs for $N$-quandles.
\end{remark}

As with $n$-quandles, we are interested in determining which links have finite $N$-quandles.  Inspired by Hoste and Shanahan \cite{HS2}, we conjecture these are the labeled links which are the singular locus for a spherical orbifold.  According to Dunbar \cite{DU}, these are the links in Tables \ref{linktable} and \ref{linktable2}.  The links in Table \ref{linktable} are exactly those with finite $n$-quandles; those in Table \ref{linktable2} have different labels on some components.  In the remainder of this paper, we will prove half of the Main Conjecture by proving that for each of these links and values of $N$, the $N$-quandle is finite.  We will do this by explicitly computing the quandles.

\begin{table}[htbp]
{$
\begin{array}{ccc}
\includegraphics[width=1.5in,trim=0 0 0 0,clip]{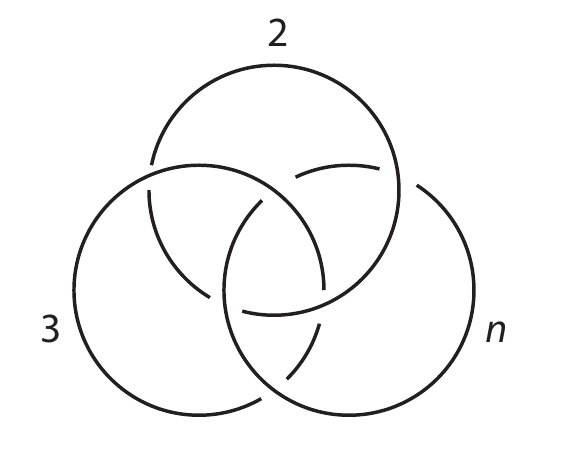}   & \hspace{.5in}& \includegraphics[width=1.25in,trim=0 0 0 0,clip]{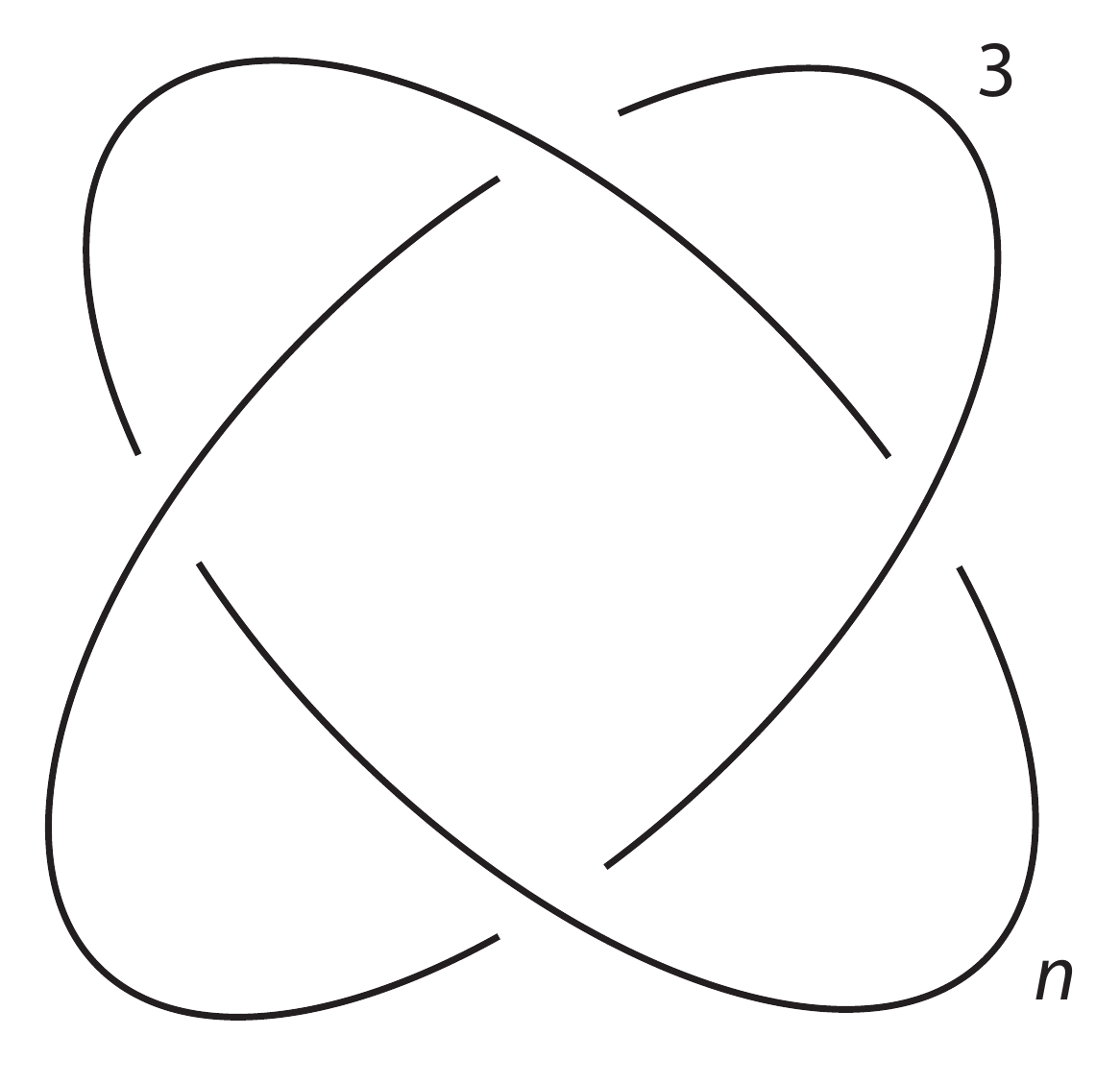}  \\
\scriptstyle L = T_{3,3};\ N = (2, 3, n);\ n = 3, 4, 5 && \scriptstyle L = T_{2,4}; N = (3, n);\ n=3, 4, 5 \\
\\
\includegraphics[width=1.25in,trim=0 0 0 0,clip]{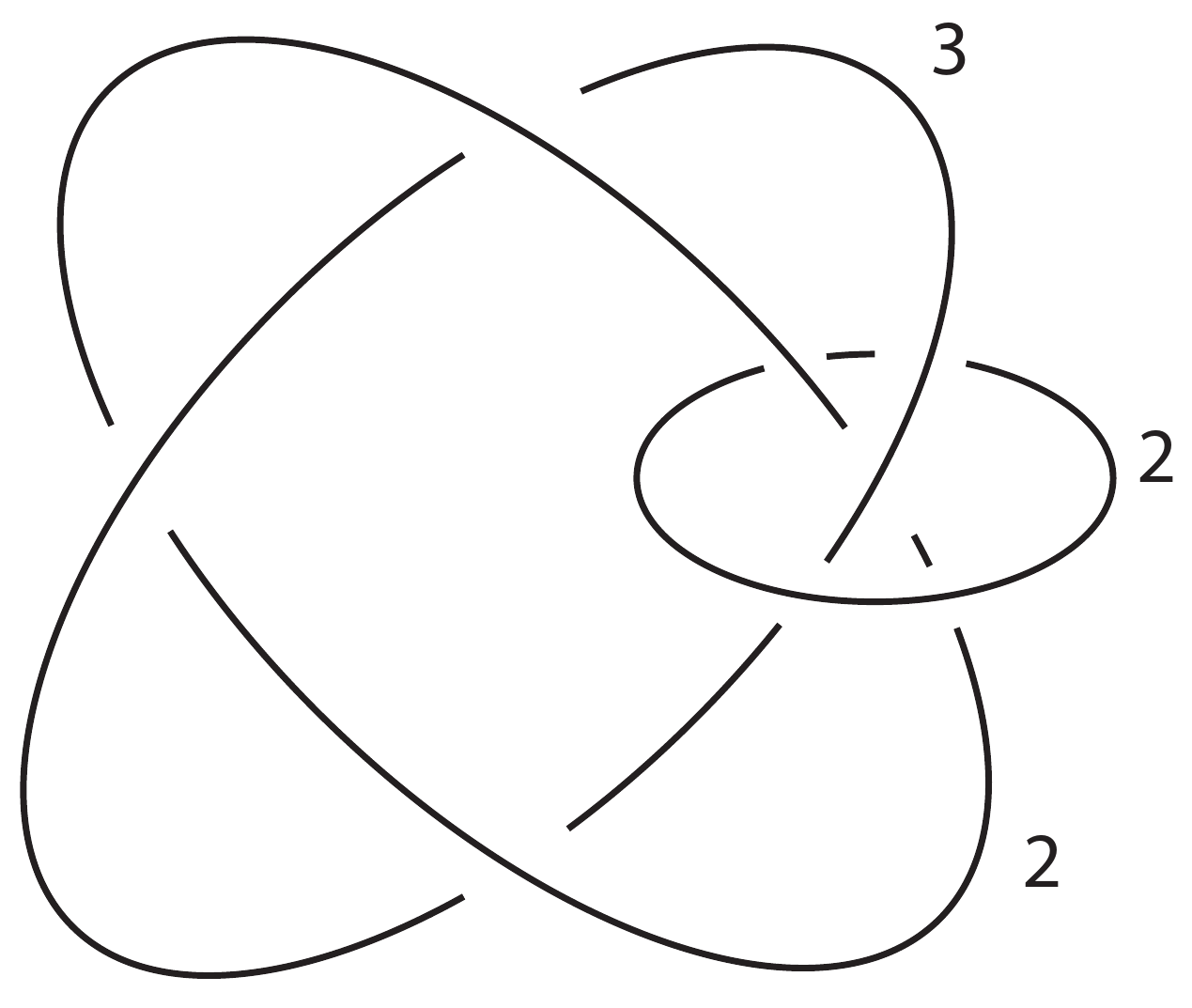} && \includegraphics[width=1.5in,trim=0 0 0 0,clip]{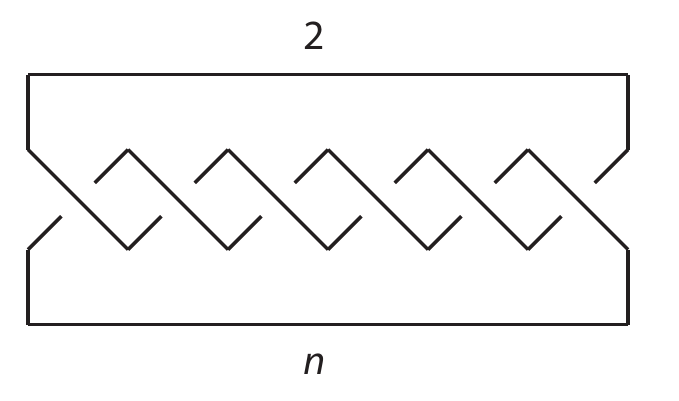}  \\
\scriptstyle L = T_{2,4} \cup C; N=(2,2,3) && \scriptstyle L = T_{2,6}; N = (2, n);\ n=3, 4, 5 \\
\\
 \includegraphics[width=1.75in,trim=0 0 0 0,clip]{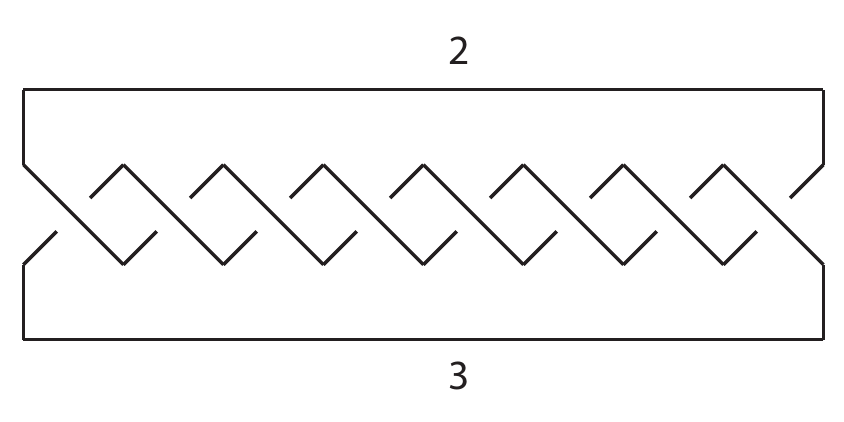}  && \includegraphics[width=2in,trim=0 0 0 0,clip]{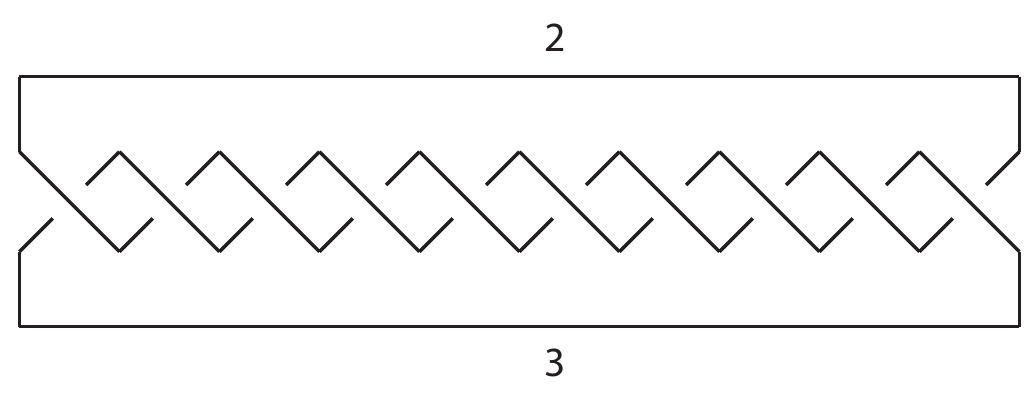} \\
\scriptstyle L = T_{2,8}; N = (2, 3) && \scriptstyle L = T_{2,10}; N = (2, 3) \\
\includegraphics[width=1.5in,trim=0 0 0 0,clip]{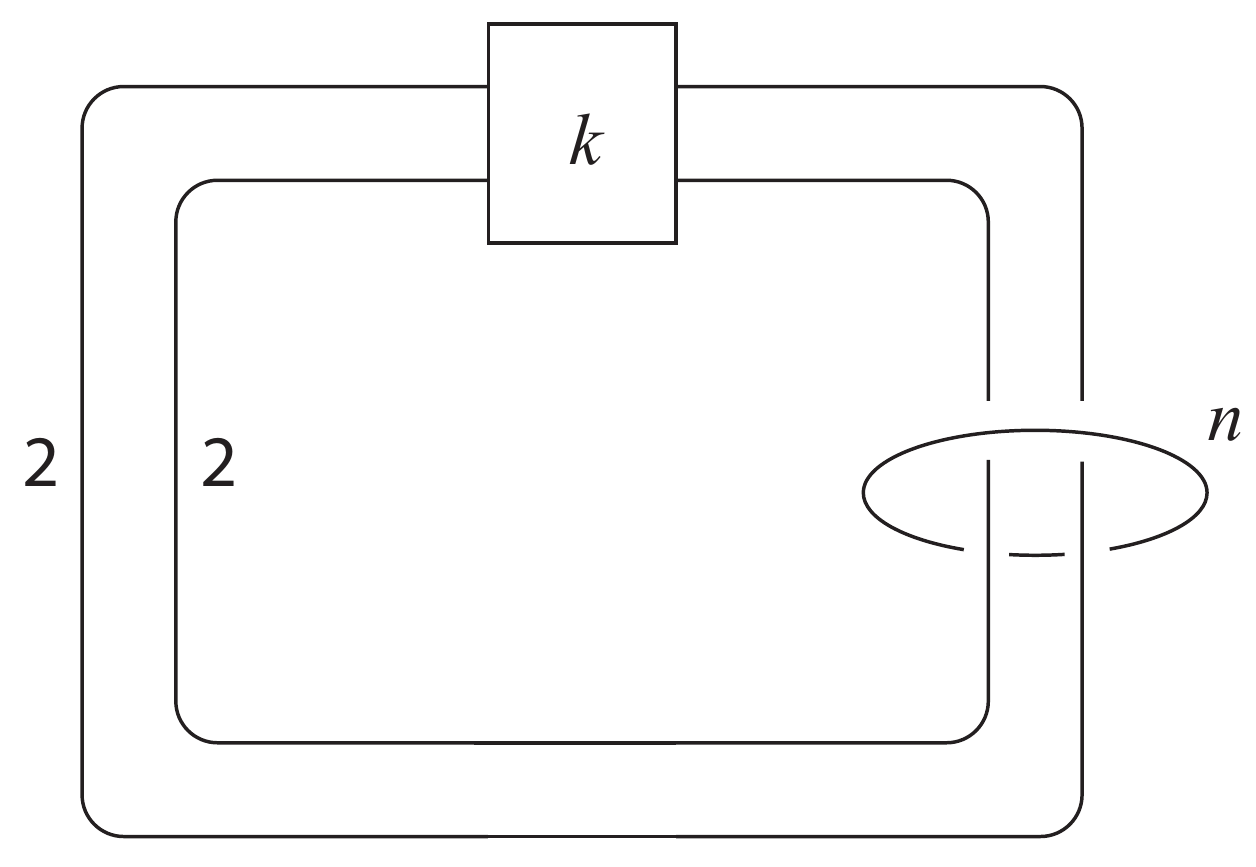}  && \includegraphics[width=1.5in,trim=0 0 0 0,clip]{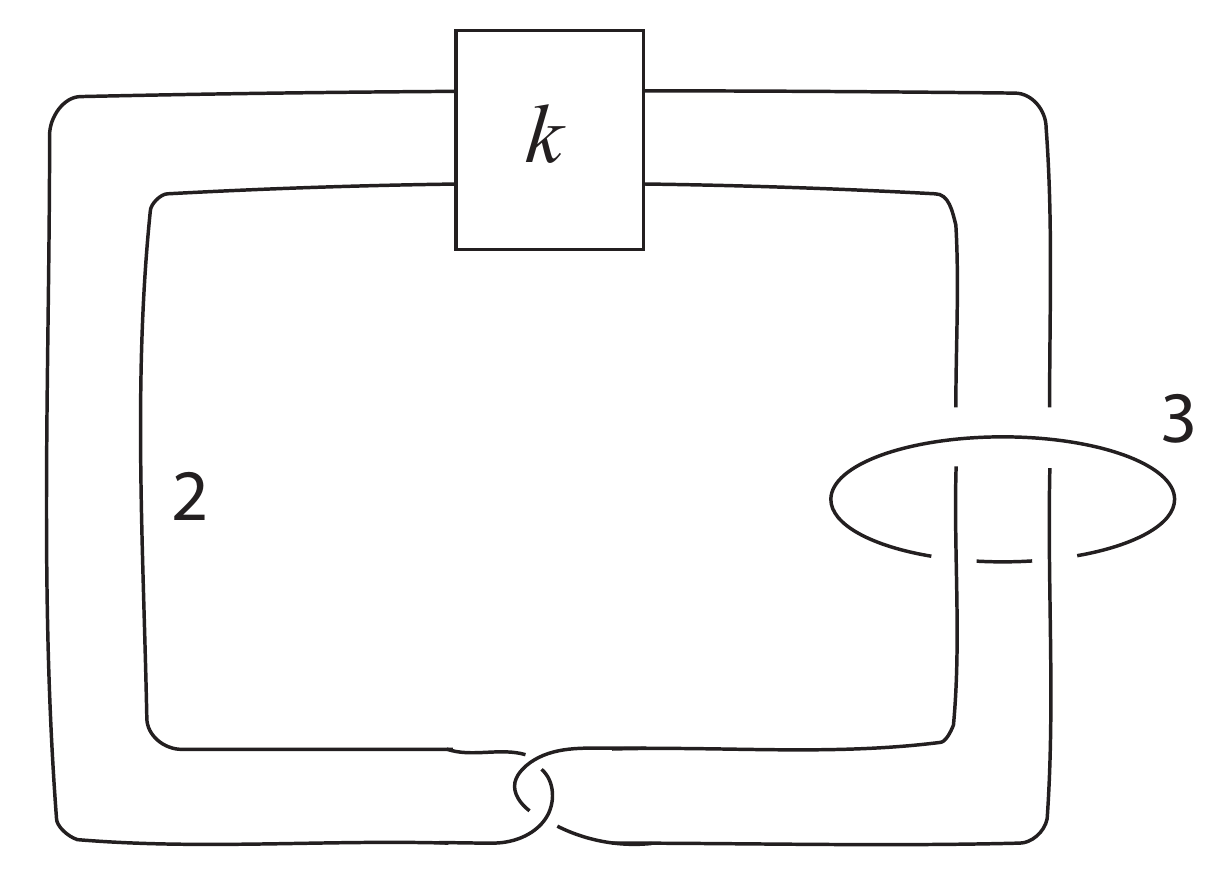}  \\
\scriptstyle L_k = T_{2,k} \cup C; N = (2, n)\ {\rm or}\ (2, 2, n);\ n > 1;\ k \neq 0 && M_k = T_k \cup C; \scriptstyle N = (2, 3)
\end{array}
$}
\caption{Links $L \in \mathbb{S}^3$ with finite $Q_N(L)$.}
\label{linktable2}
\end{table}

\section{Summary of results} \label{S:summary}

In Table \ref{quandletable} we list the cardinalities of all finite $N$-quandles (including $n$-quandles) for the links in Tables \ref{linktable} and \ref{linktable2}; except for the links in the last row of Table \ref{linktable}, which have not yet been determined. These cardinalities have been found by explicitly describing the Cayley graphs for the $N$-quandles. Some of these results were found in earlier papers \cite{CHMS, HS1, Me}; we have given the citations in the table. In most cases (all except the first two and the last four), we are considering a specific $N$-quandle for a specific link; the Cayley graphs were computed explicitly using {\em Mathematica}.  In section \ref{S:Cayley} we will describe the algorithm we use to compute the Cayley graphs. The presentations and Cayley graphs for these $N$-quandles are given in Section \ref{S:graphs}.  The {\em Mathematica} code used to produce them can be downloaded from the first author's website \cite{Me2}. (We are currently working to translate the program into Python; once we have done that we will post the Python code as well, and upload it to the arXiv along with this paper.)

\begin{table}[htbp]
{
\begin{tabular}{|c|c|c|c|}
\hline
Link $L$ & $N$ & $\vert Q_N(L) \vert$ & Reference \\ \hline
$T_{2,k}$ & (2) or (2,2) & $k$ & Crans, et. al. \cite{CHMS}  \\ \hline
$T_{2,k} \cup C$, $k \neq 0$ & $(2,n)$ or $(2,2,n)$, $n \geq 2$ & $2 + n\vert k\vert$ & Theorem \ref{T:QNLk} \\ \hline
\multirow{3}{*}{$T_{2,3}$} & (3) & 4 & \multirow{3}{*}{Crans, et. al. \cite{CHMS}} \\ \cline{2-3}
& (4) & 6 & \\ \cline{2-3}
& (5) & 12 &  \\ \hline
$T_{2,3} \cup B$ & (2,2) & 18 & Crans, et. al. \cite{CHMS} \\ \hline
\multirow{3}{*}{$T_{2,4}$} & (3,3) & 8 & \multirow{3}{*}{Section \ref{S:graphs}} \\ \cline{2-3}
& (3,4) & 14 & \\ \cline{2-3}
& (3,5) & 32 & \\ \hline
$T_{2,4} \cup C$ & (2,2,3) & 26 & Section \ref{S:graphs} \\ \hline
$T_{2,5}$ & (3) & 20 & Crans, et. al. \cite{CHMS} \\ \hline
\multirow{3}{*}{$T_{2,6}$} & (2,3) & 10 & \multirow{3}{*}{Section \ref{S:graphs}} \\ \cline{2-3}
& (2,4) & 18 & \\ \cline{2-3}
& (2,5) & 42 & \\ \hline
$T_{2,8}$ & (2,3) & 20 & Section \ref{S:graphs} \\ \hline
$T_{2,10}$ & (2,3) & 50 & Section \ref{S:graphs} \\ \hline
\multirow{4}{*}{$T_{3,3}$} & (2,2,2) & 6 & Crans, et. al. \cite{CHMS} \\ \cline{2-4}
& (2,3,3) & 14 & \multirow{3}{*}{Section \ref{S:graphs}} \\ \cline{2-3}
& (2,3,4) & 26 & \\ \cline{2-3}
& (2,3,5) & 62 & \\ \hline
$T_{3,4}$ & (2) & 12 & Crans, et. al. \cite{CHMS} \\ \hline
$T_{3,5}$ & (2) & 30 & Crans, et. al. \cite{CHMS} \\ \hline
$L_{p/q}$ & (2) or (2,2) & $q$ & Crans, et. al. \cite{CHMS} \\ \hline
$L_{k,p/q} \cup C$ & (2,2) or (2,2,2) & $2q(\vert kq-p\vert +1)$ & Mellor \cite{Me} \\ \hline
$L(1/2,1/2,p/q;k)$ & (2,2) or (2,2,2) & $2(q+1)\vert (k-1)q-p \vert$ & Hoste and Shanahan \cite{HS1} \\ \hline
$T_k \cup C$ & (2,3) & $18\vert 2k-1\vert + 8$ & Theorem \ref{T:QNMk} \\ \hline
\end{tabular}
}
\caption{Finite $N$-quandles of links.}
\label{quandletable}
\end{table}

The two families of links $L_k = T_{2,k}\cup C$ and $M_k = T_k \cup C$ require more analysis.  In Section \ref{SS:Lk} we prove

\begin{theorem}\label{T:QNLk}
Let $L_k = T_{2,k}\cup C$, and assume $n \geq 2$.
\begin{enumerate}
	\item If $k$ is odd, then $\vert Q_{(2,n)}(L_k) \vert = n\vert k\vert + 2$.
	\item If $k$ is even, then $\vert Q_{(2,2,n)}(L_k) \vert = n\vert k \vert + 2$.
\end{enumerate}
\end{theorem}

In Section \ref{SS:Mk} we prove

\begin{theorem}\label{T:QNMk}
Let $M_k = T_k \cup C$.  Then $\vert Q_{(2,3)}(M_k) \vert = 18\vert 2k-1\vert + 8$.
\end{theorem}

This completes the justification of Table \ref{quandletable}, and proves one direction of the Main Conjecture.

\begin{theorem}\label{T:oneway}
If $L$ is a $k$-component link which is the singular locus of a spherical orbifold with underlying space $\mathbb{S}^3$, with component $i$ labeled $n_i$, then the quandle $Q_{(n_1,\dots,n_k)}(L)$ is finite.
\end{theorem}

\section{Computing Cayley graphs} \label{S:Cayley}

Given a presentation of a quandle, one can try to systematically enumerate its elements and simultaneously produce a Cayley graph of the quandle. Such a method was described in a graph-theoretic fashion by Winker in \cite{WI}. The method is similar to the well-known Todd-Coxeter process for enumerating cosets of a subgroup of a group \cite{TC} and has been extended to racks by Hoste and Shanahan \cite{HS3}. (A rack is more general than a quandle, requiring only axioms A2 and A3.) We provide a brief description of Winker's method applied to the $N$-quandle of a link.  Suppose $Q_N(L)$ is presented as
$$Q_N(L)=\left\langle x_1, x_2, \dots, x_g \, \mid\, x_{j_1}^{w_1}=x_{k_1}, \dots, x_{j_r}^{w_r}=x_{k_r}, \left\{x_i^{x_j^{n_j}} = x_i\right\}_{i, j = 1}^g \right\rangle,$$
where each $w_i$ is a word in $\{x_1,\dots , x_g, \overline{x_i}, \dots, \overline{x_g}\}$, and $n_j$ is the label on the quandle component containing $x_j$. As noted in Remark \ref{R:Nquandle}, the set of relations $\left\{x_i^{x_j^{n_j}}= x_i\right\}_{i = 1}^g$ implies $x^{x_j^{n_j}} = x$ for any element $x$ of the quandle. For convenience, throughout this paper presentations of $N$-quandles will list the single relation $x^{x_j^{n_j}} = x$ in place of $\left\{x_i^{x_j^{n_j}}= x_i\right\}_{i = 1}^g$; $x$ should be understood to be any element of the quandle. 

If $y$ is any element of the quandle, then it follows from the relation $x_{j_i}^{w_i}=x_{k_i}$ and Lemma~\ref{leftassoc} that $y^{\overline{w}_i x_{j_i}w_i}=y^{x_{k_i}}$, and so
$$y^{\overline{w}_i x_{j_i} w_i \overline{x}_{k_i}}=y \text{ for all } y \text{ in } Q_N(L).$$

Winker calls this relation the {\it secondary relation}  associated to the {\it primary relation} $x_{j_i}^{w_i}=x_{k_i}$. We also consider relations of the form $y^{ x_j^{n_j} }=y$ for all $y$ and $1 \le j \le g$. These relations are equivalent to the secondary relations of the $N$-quandle relations \cite{HS2}.   

Winker's method now proceeds to build the Cayley graph associated to the presentation as follows:
  
\begin{enumerate}
\item Begin with $g$ vertices labeled $x_1,x_2, \dots, x_g$ and numbered $1,2, \dots,g$. 
\item Add an oriented loop at each vertex $x_i$ and label it $x_i$. (This encodes the axiom A1.)
\item For each value of $i$ from $1$ to $r$, {\em trace} the primary relation $x_{j_i}^{w_i}=x_{k_i}$ by introducing new vertices and oriented edges as necessary to create an oriented path from $x_{j_i}$ to $x_{k_i}$ given by $w_i$. Consecutively number (starting from $g+1$) new vertices in the order they are introduced.  Edges are labelled with their corresponding generator and oriented to indicate whether $x_i$ or $\overline x_i$ was traversed. 
\item Tracing a relation may introduce edges with the same label and same orientation into or out of a shared vertex. We identify all such edges, possibly leading to other identifications. This process is called {\it collapsing} and all collapsing is carried out before tracing the next relation. 
\item Proceeding in order through the vertices, trace and collapse each $N$-quandle relation $y^{x_j^{n_j}}=y$  and each secondary relation (in order). All of these relations are traced and collapsed at a vertex before proceeding to the next vertex.
\end{enumerate}

The method will terminate in a finite graph if and only if the $N$-quandle is finite. The reader is referred to Winker~\cite{WI} and Hoste and Shanahan \cite{HS3}  for more details.

\section{Cayley Graphs for specific quandles} \label{S:graphs}

This section contains the quandle presentations and Cayley graphs for the $N$-quandles discussed in Section \ref{S:summary} which are not members of infinite families (and which were not computed in \cite{CHMS}).  In the graphs, given generators $a$, $b$ and $c$, solid lines connect $x$ to $x^a$, dashed lines connect $x$ to $x^b$, and dotted lines connect $x$ to $x^c$ (if there is a third generator).  If $x^a = x^{\bar{a}}$, then the edge is not oriented (and similarly with the other generators). In the quandle presentations ``$x$" stands for an arbitrary element of the quandle (or, without loss of generality, for any of the generators).

\subsection{$T_{2,4}$}

\parbox{5in}{$Q_{(3,4)}(T_{2,4}) = \langle a, b \mid a^{bab} = a,\ b^{aba} = b,\ x^{a^3} = x^{b^4} = x \rangle$

$$\scalebox{.5}{\includegraphics{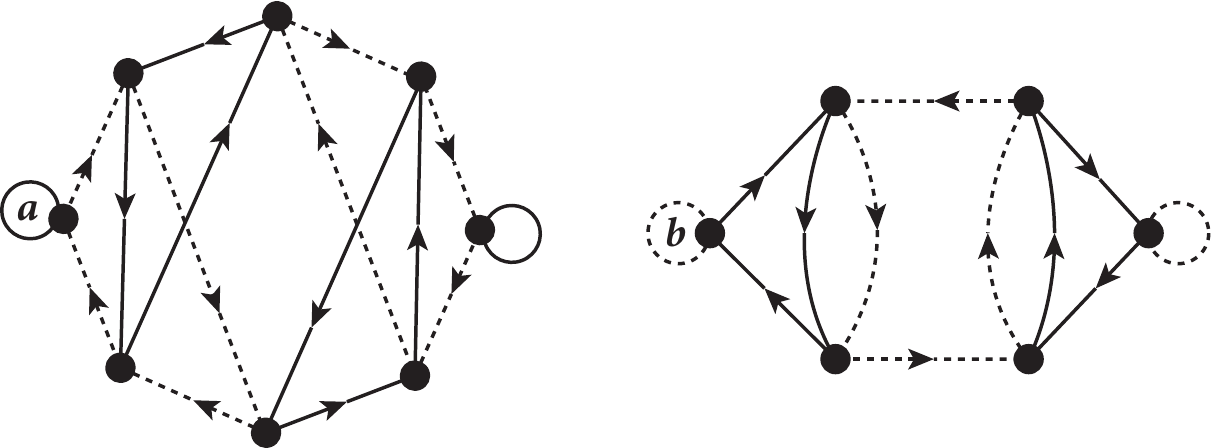}}$$}

\noindent\parbox{5in}{$Q_{(3,5)}(T_{2,4}) = \langle a, b \mid a^{bab} = a,\ b^{aba} = b,\ x^{a^3} = x^{b^5} = x \rangle$

$$\scalebox{.5}{\includegraphics{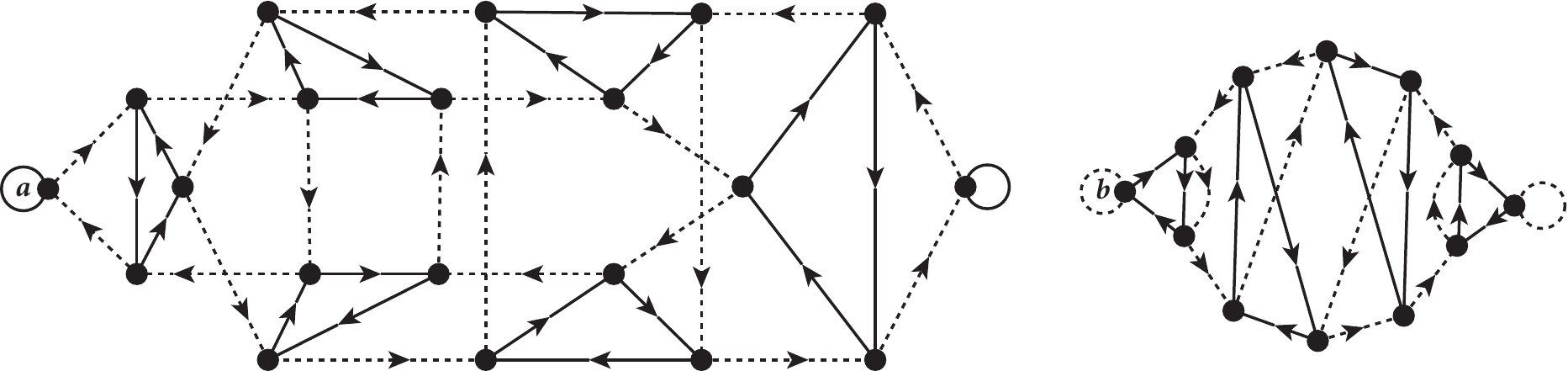}}$$}

\subsection{$T_{2,4}\cup C$}

\parbox{5in}{$Q_{(2,3,2)}(T_{2,4} \cup C) = \langle a, b, c \mid a^{babc} = a,\ b^{ababc} = b,\ c^{ab} = c,\ x^{a^2} = x^{b^3} = x^{c^2} = x \rangle$

$$\scalebox{.5}{\includegraphics{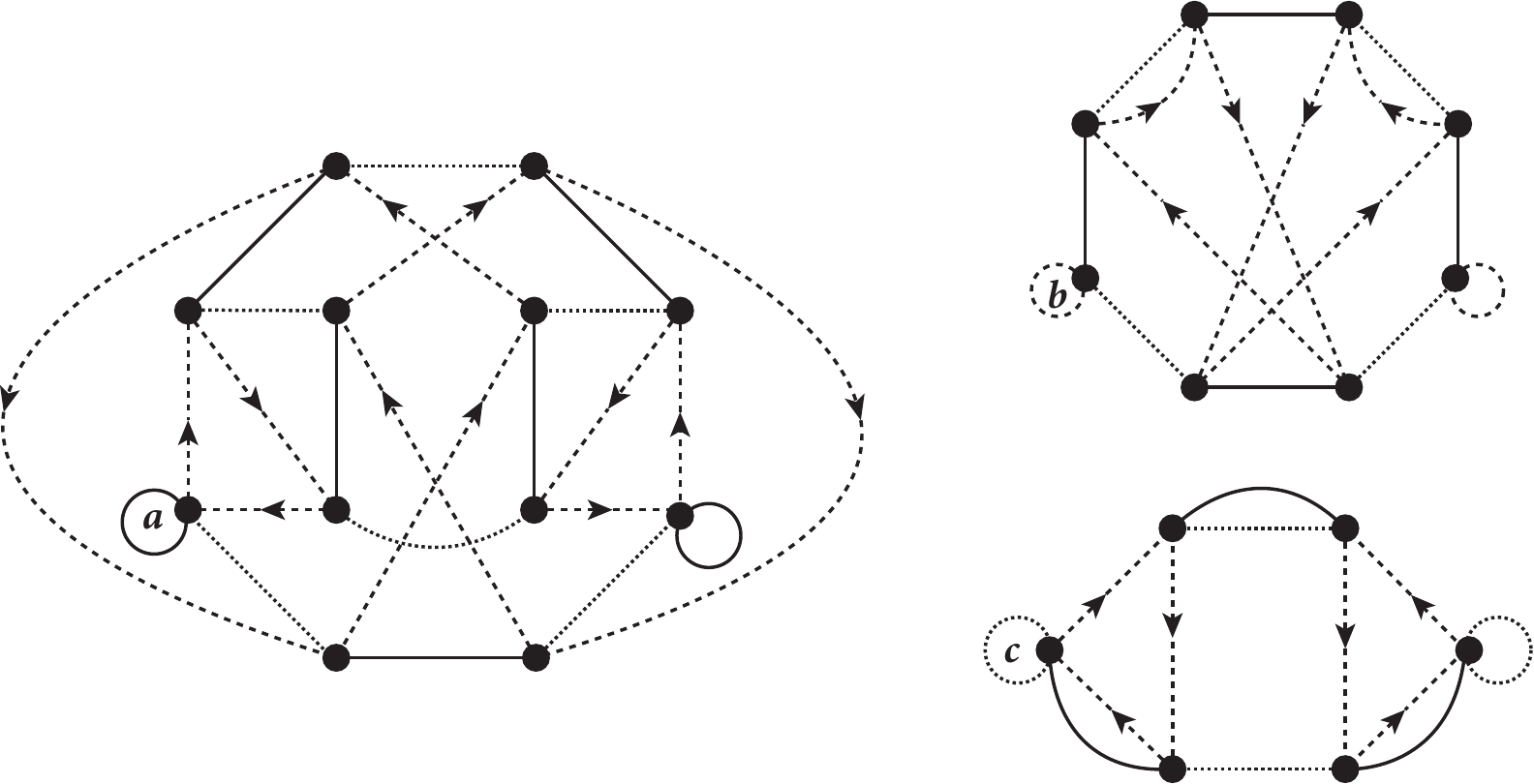}}$$}

\subsection{$T_{2,6}$}

\noindent \parbox{5in}{$Q_{(2,3)}(T_{2,6}) = \langle a, b \mid a^{babab} = a,\ b^{ababa} = b,\ x^{a^2} = x^{b^3} = x \rangle$

$$\scalebox{.5}{\includegraphics{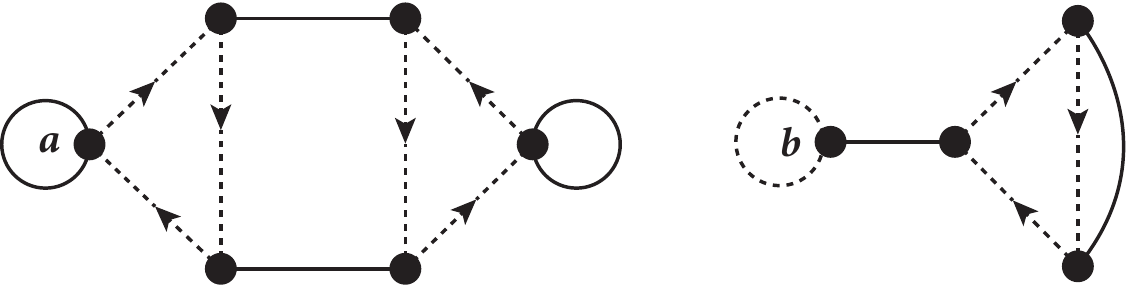}}$$}

\noindent \parbox{5in}{$Q_{(2,4)}(T_{2,6}) = \langle a, b \mid a^{babab} = a,\ b^{ababa} = b,\ x^{a^2} = x^{b^4} = x \rangle$

$$\scalebox{.5}{\includegraphics{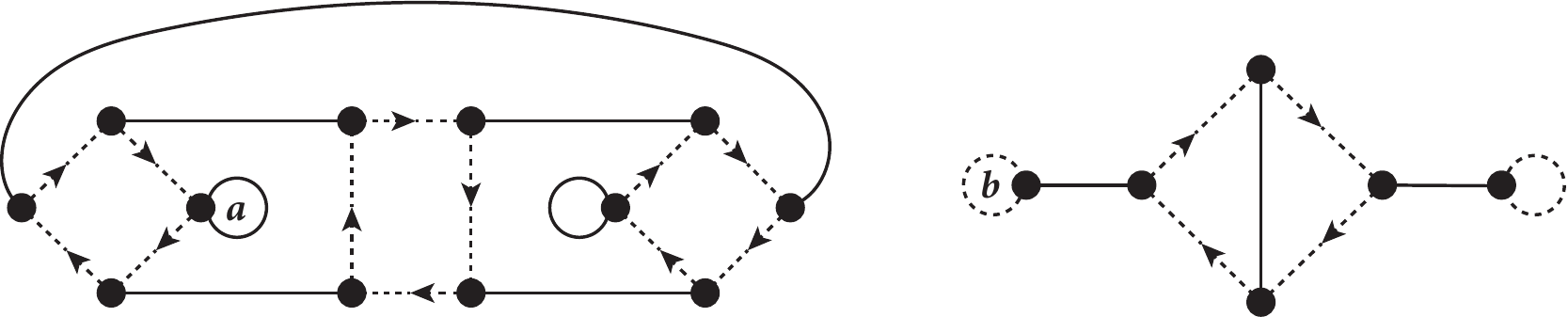}}$$}

\noindent \parbox{5in}{$Q_{(2,5)}(T_{2,6}) = \langle a, b \mid a^{babab} = a,\ b^{ababa} = b,\ x^{a^2} = x^{b^5} = x \rangle$

$$\scalebox{.5}{\includegraphics{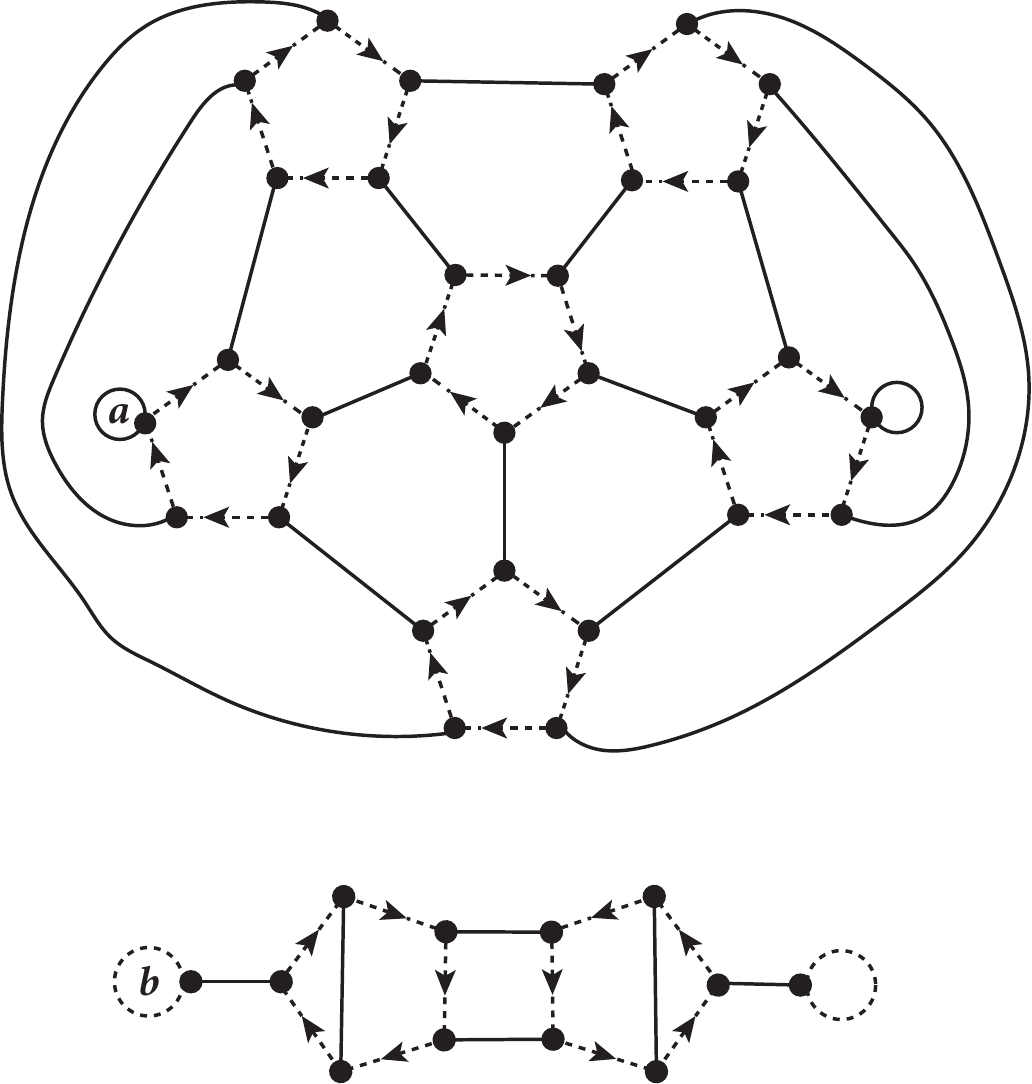}}$$}

\subsection{$T_{2,8}$}

\noindent \parbox{5in}{$Q_{(2,3)}(T_{2,8}) = \langle a, b \mid a^{bababab} = a,\ b^{abababa} = b,\ x^{a^2} = x^{b^3} = x \rangle$

$$\scalebox{.5}{\includegraphics{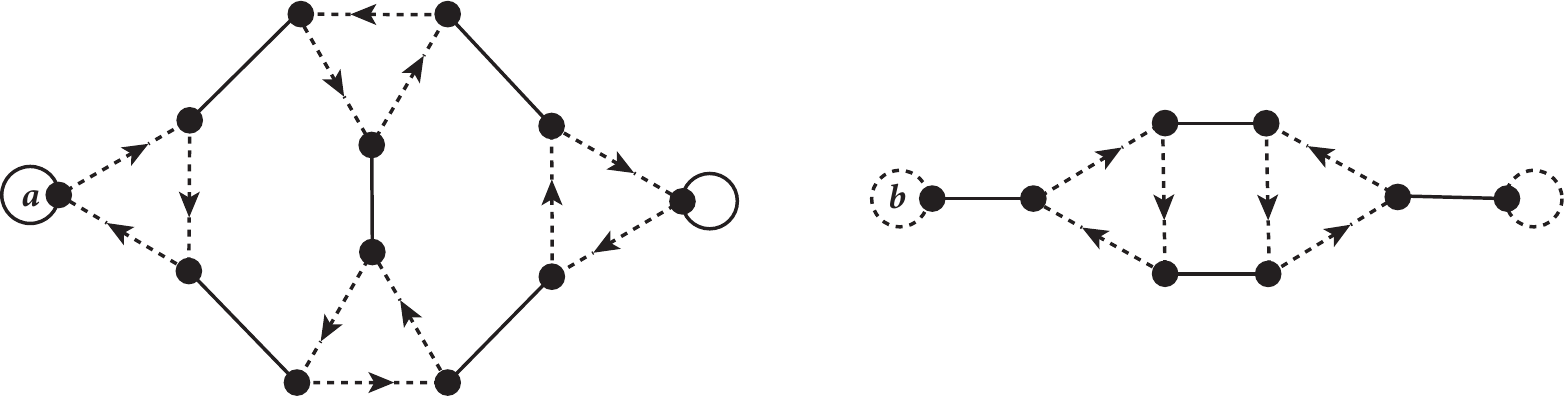}}$$}

\subsection{$T_{2,10}$}

\noindent \parbox{5in}{$Q_{(2,3)}(T_{2,10}) = \langle a, b \mid a^{babababab} = a,\ b^{ababababa} = b,\ x^{a^2} = x^{b^3} = x \rangle$

$$\scalebox{.5}{\includegraphics{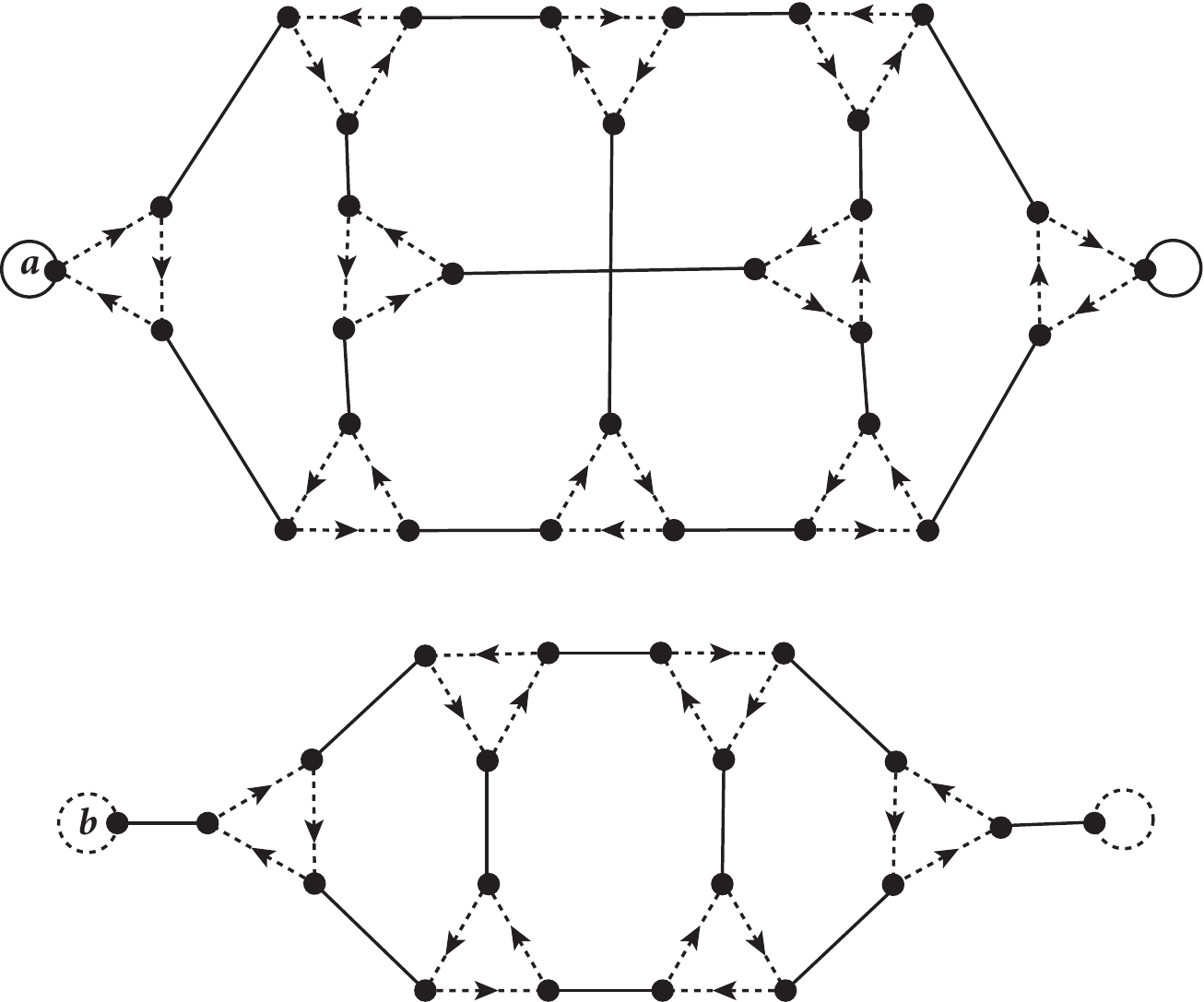}}$$}

\subsection{$T_{3,3}$}

\noindent \parbox{5in}{$Q_{(2,3,3)}(T_{3,3}) = \langle a, b, c \mid a^{cb} = a,\ b^{ac} = b,\ c^{ba} = c,\ x^{a^2} = x^{b^3} = x^{c^3} = x \rangle$

$$\scalebox{.6}{\includegraphics{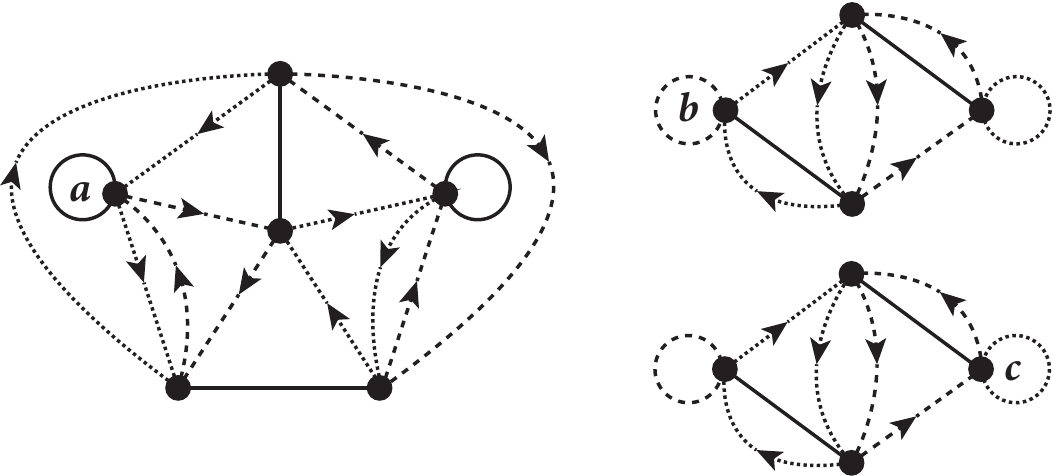}}$$}

\noindent \parbox{5in}{$Q_{(2,3,4)}(T_{3,3}) = \langle a, b, c \mid a^{cb} = a,\ b^{ac} = b,\ c^{ba} = c,\ x^{a^2} = x^{b^3} = x^{c^4} = x \rangle$

$$\scalebox{.4}{\includegraphics{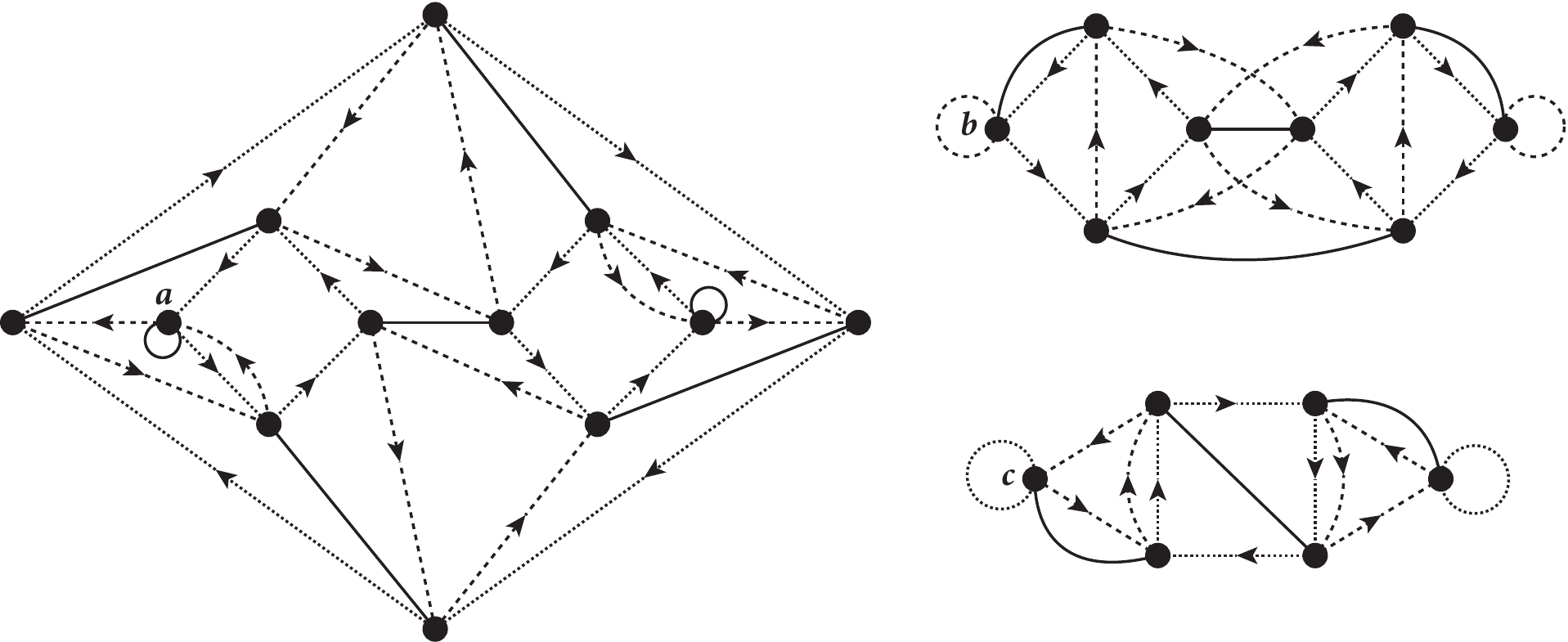}}$$}

\noindent \parbox{5in}{$Q_{(2,3,5)}(T_{3,3}) = \langle a, b, c \mid a^{cb} = a,\ b^{ac} = b,\ c^{ba} = c,\ x^{a^2} = x^{b^3} = x^{c^5} = x \rangle$

$$\scalebox{.45}{\includegraphics{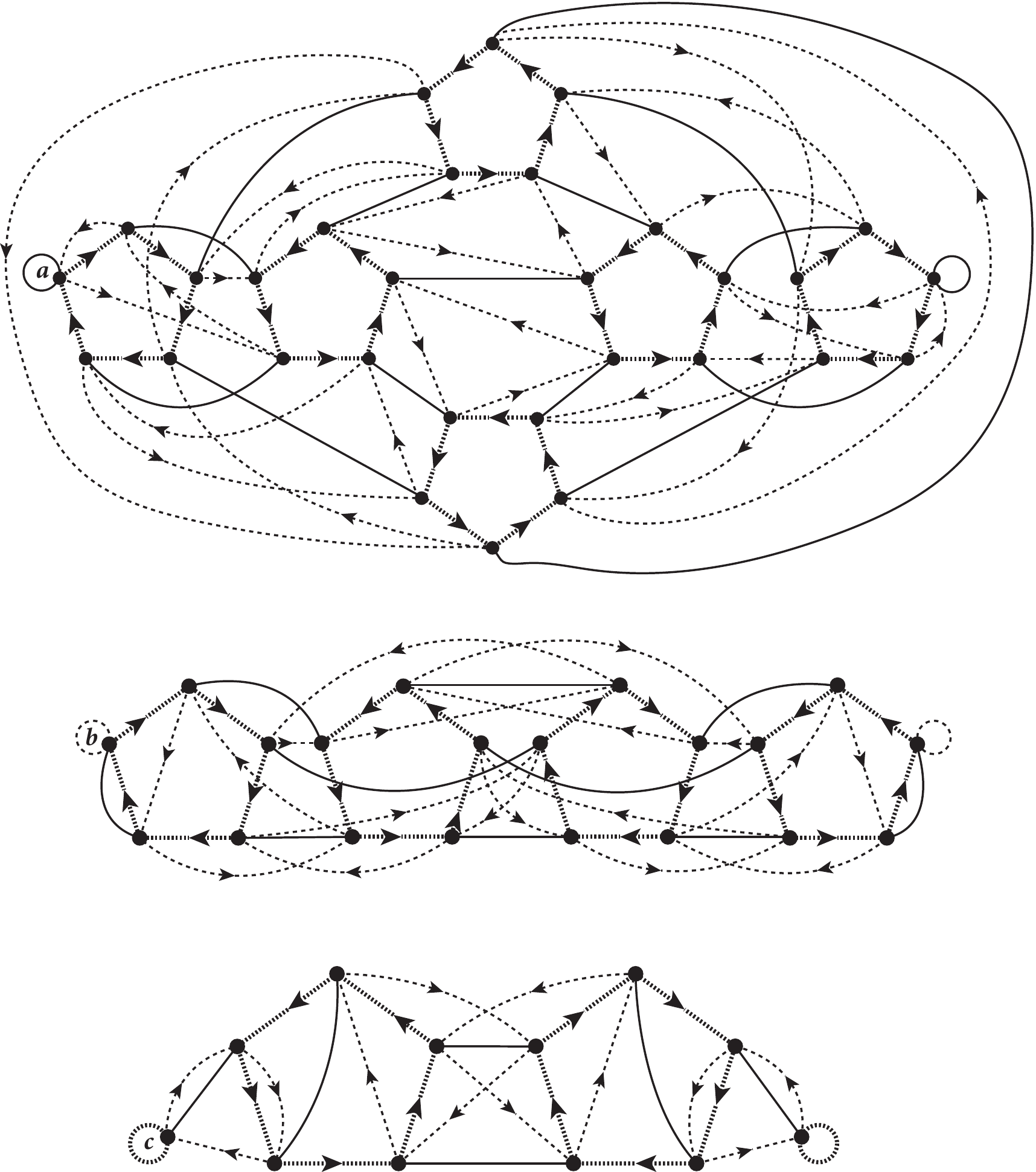}}$$}

\section{The family of links $L_k = T_{2,k} \cup C$} \label{SS:Lk}

We now consider the family of links $L_k = T_{2,k} \cup C$, with $k \neq 0$, shown in Figure \ref{F:Lk}. Here $k$ represents the number of right-handed half-twists (if $k$ is negative, the twists are left-handed); one such half-twist is shown.  The link has two components if $k$ is odd and three components if $k$ is even. We will construct the finite Cayley graph for $Q_N(L_k)$ when $N = (2,n)$ (for $k$ odd) or $N = (2,2,n)$ (for $k$ even), for any integer $n > 1$ (i.e. the label on the torus knot or link is 2, and the label on the other component $C$ is $n$).  As a consequence of our construction, we prove Theorem \ref{T:QNLk}.  The case when $n = 2$ was dealt with by Crans et. al. \cite{CHMS}, so we will assume $n \geq 3$.

\begin{figure}[h]
$$\includegraphics[height=1.5in]{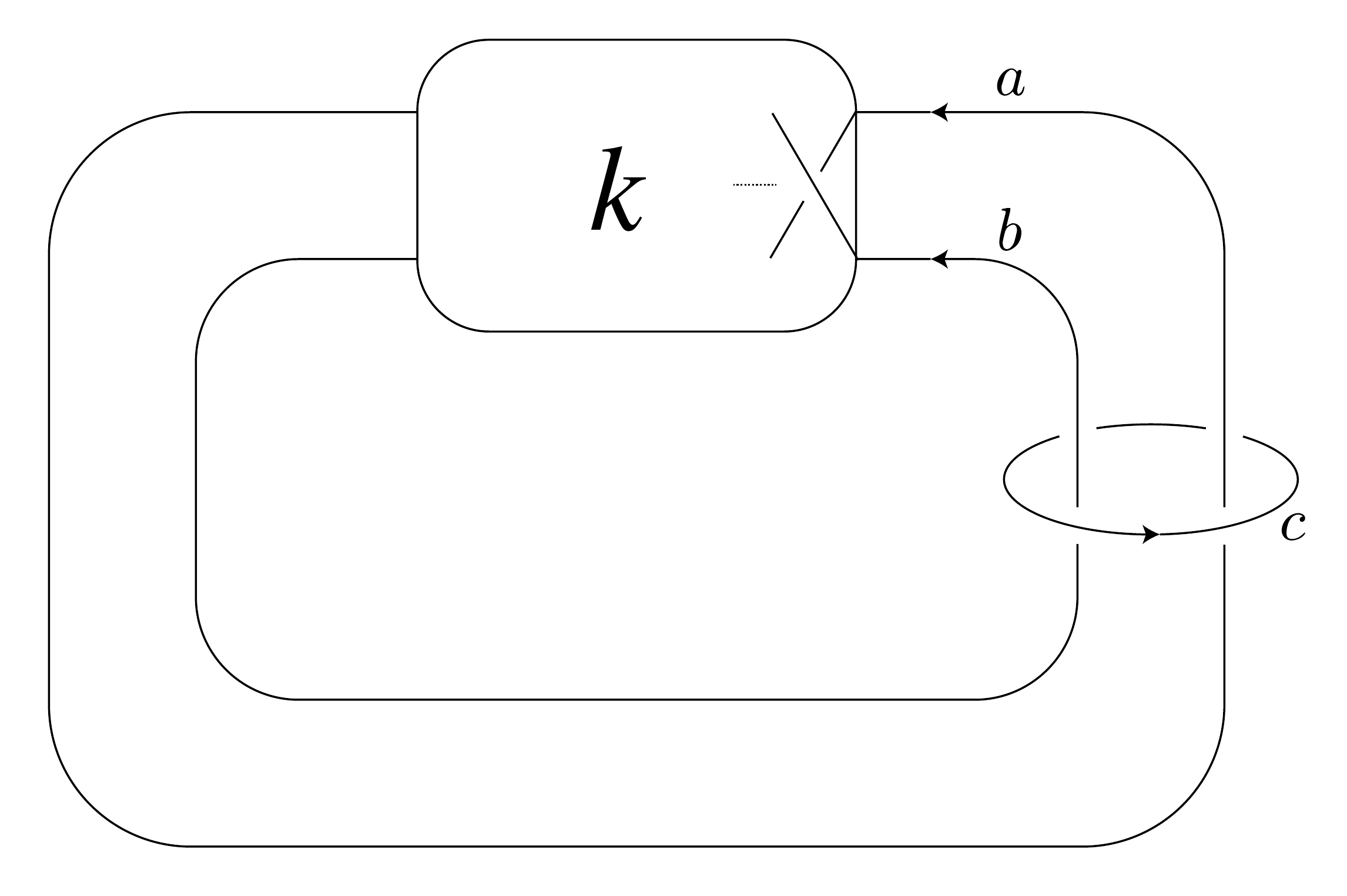}$$
\caption{$L_k = T_{2,k} \cup C$; $k \neq 0$}
\label{F:Lk}
\end{figure}

\subsection{$L_k$ with $k$ odd} \label{SSS:odd}

If we orient the link as shown in Figure \ref{F:Lk}, the quandle has three generators $a, b, c$.  There are two cases, depending on whether $k$ is odd or even.  If $k = 2t+1$ is odd, then the quandle $Q_{(2,n)}(L_k)$ has the following presentation (this can be seen through an easy inductive argument), where $x$ can be any of the generators $a, b, c$:
$$Q_{(2,n)}(L_k) = \langle a, b, c \mid c^{ab} = c,\ a^{(ba)^t bc} = b,\ b^{(ab)^t c} = a, x^{a^2} = x^{b^2} = x^{c^n} = x \rangle.$$
We will denote the relations $R_1$, $R_2$, $R_3$ and $J_{xy}$ (where $J_{xy}$ is the relation $x^{y^{n_y}} = x$, for $x, y \in \{a, b, c\}$).

\begin{remark} \label{R:k<0}
These relations also hold if $t$ is negative, with the convention that $x^{y^{-n}} = x^{\bar{y}^n}$.  Observe that if $t \geq 0$ and $k = -(2t+1) = 2(-t-1)+1$, then the relations become $a^{(ba)^{-t-1} bc} = a^{(ab)^{t+1} bc} = a^{a(ba)^t bbc} = a^{(ba)^t c} = b$ and $b^{(ab)^{-t-1} c} = b^{(ba)^{t+1} c} = b^{(ab)^t ac} = a$.  These are the same as the relations for $k = 2t+1$, except that the generators $a$ and $b$ are switched.  Hence $Q_{(2,n)}(L_{-k})$ is isomorphic to $Q_{(2,n)}(L_k)$.
\end{remark}

We will construct Cayley graphs for these quandles (in light of Remark \ref{R:k<0}, we will assume $k > 0$ in the graphs). Figure \ref{F:Q(2,n)Lkodd} shows the Cayley graphs for $Q_{(2,5)}(L_k)$ and $Q_{(2,6)}(L_k)$; it is easy to see how the pattern continues for larger odd and even values of $n$. In these graphs, the dotted edges represent the operation of the generator $c$.

\begin{figure}[htbp]
$$\includegraphics[height=3.5in]{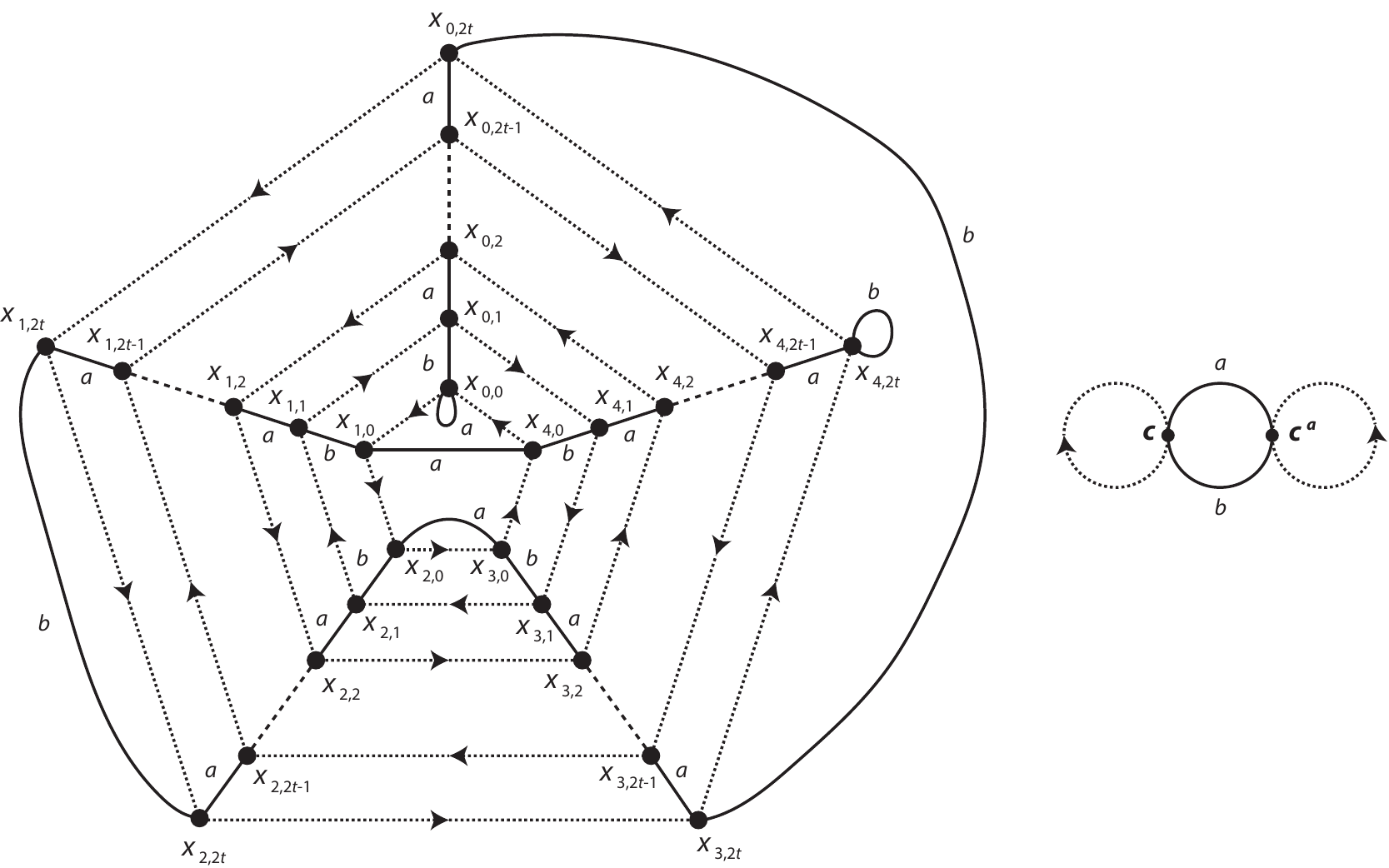}$$
$$\includegraphics[height=3.5in]{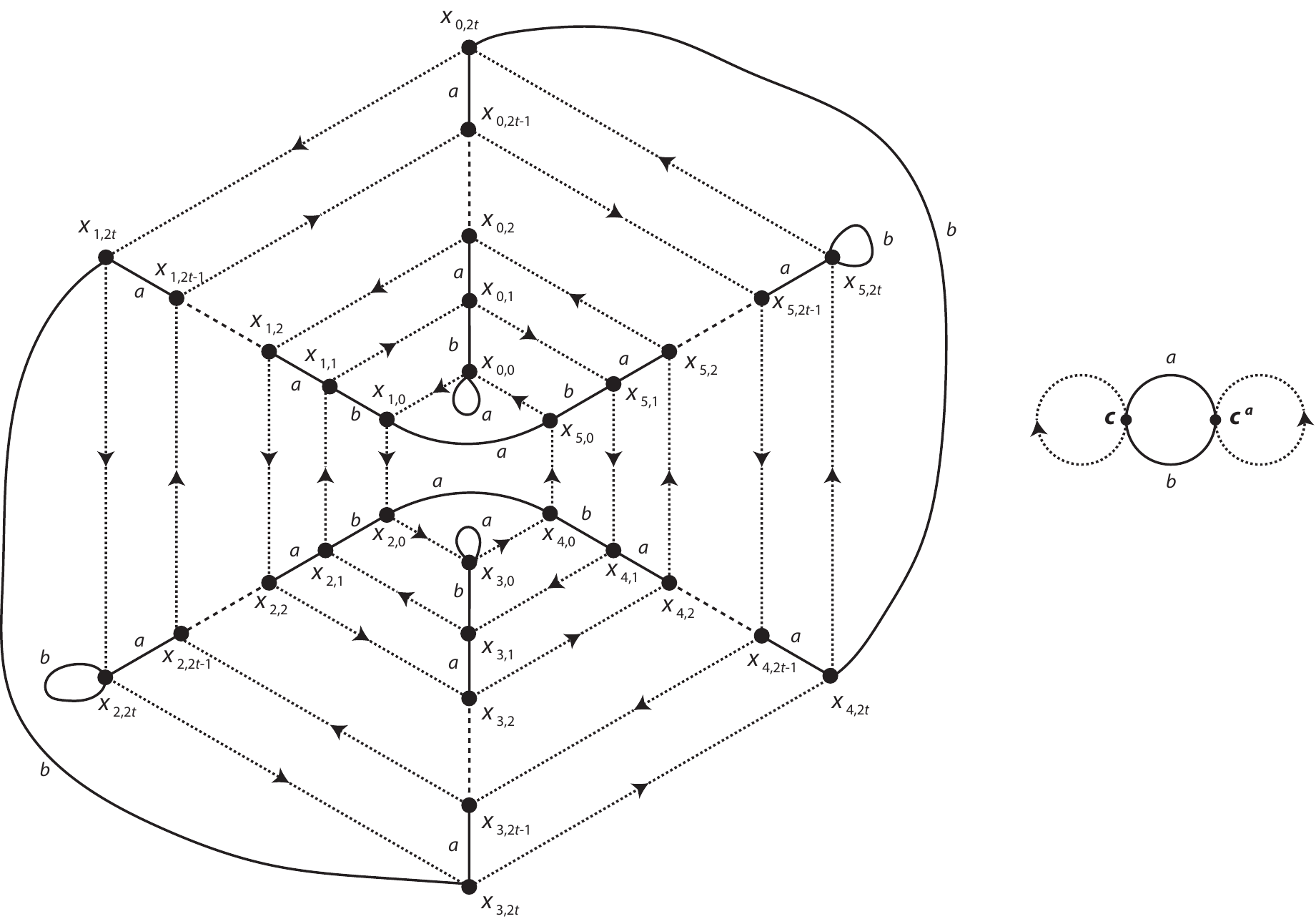}$$
\caption{$Q_{(2,5)}(L_k)$ and $Q_{(2,6)}(L_k)$ for $k = 2t+1$; dotted edges represent generator $c$.}
\label{F:Q(2,n)Lkodd}
\end{figure}

We begin by deriving an alternative presentation for $Q_{(2,n)}(L_k)$ which makes Winker's algorithm easier to use.  We first observe some useful consequences of relation $R_1: c^{ab} = c$ (these were also observed in \cite{Me}).

\begin{lemma}\label{L:r1}
For any $x \in Q_{(2,n)}(L_k)$, we have (here $\tilde{c}$ represents either $c$ or $\bar{c}$)
\begin{enumerate}
	\item $x^{a\tilde{c}a} = x^{b\tilde{c}b}$.
	\item $x^{\tilde{c}ab} = x^{ab\tilde{c}}$ and $x^{\tilde{c}ba} = x^{ba\tilde{c}}$.
	\item For all $w \in \{\tilde{c}a, a\tilde{c}, \tilde{c}b, b\tilde{c}\}$, $x^{wab} = x^{baw}$ and $x^{wba} = x^{abw}$.
	\item For all $v, w \in \{\tilde{c}a, a\tilde{c}, \tilde{c}b, b\tilde{c}\}$, $x^{vwab} = x^{abvw}$ and $x^{vwba} = x^{bavw}$.
\end{enumerate}
\end{lemma}
\begin{proof}
Since $c^a = c^b$, we immediately get $x^{aca} = x^{c^a} = x^{c^b} = x^{bcb}$.  Then $x^{(aca)(b\bar{c}b)} = x$ for any $x$.  In particular, $x^{a\bar{c}a} = x^{(a\bar{c}a)(aca)(b\bar{c}b)} = x^{b\bar{c}b}$.  We will prove the remaining relations for $c$; the same arguments work for $\bar{c}$.

For part (2), $x^{cab} = x^{a(aca)b} = x^{a(bcb)b} = x^{abc}$.  Similarly, $x^{cba} = x^{b(bcb)a} = x^{b(aca)a} = x^{bac}$.

For part (3), we consider the case when $w = ca$.  Then, using part (1), $x^{caab} = x^{cb} = x^{bbcb} = x^{baca}$ and $x^{caba} = x^{aacaba} = x^{abcbba} = x^{abca}$.  The other cases are proved similarly.  Part (4) is just the application of (3) twice.
\end{proof}

Now we will derive several other relations in $Q_{(2,n)}(L_k)$.

\begin{lemma} \label{L:relations}
The following relations hold in $Q_{(2,n)}(L_k)$.
\begin{enumerate}
	\item $P_1: a^{\bar{c}(ba)^t} = b$.
	\item $P_2: a^{(ba)^t\bar{c}} = b$.
	\item $P_3: a^{cac} = a$.
	\item $P_4: b^{cbc} = b$.
\end{enumerate}
\end{lemma}
\begin{proof}
$P_1$ is equivalent to $R_3$ by the second quandle axiom.  $P_2$ follows from $P_1$ by Lemma \ref{L:r1}.

From relation $R_2$, $b = a^{(ba)^t b c} = a^{(ab)^{t+1} c} = a^{c(ab)^{t+1}}$ (by Lemma \ref{L:r1}).  Hence $a^c = b^{(ba)^{t+1}} = b^{(ab)^t a} = a^{\bar{c} a}$ by relation $R_3$.  This implies $a^{cac} = a$, giving relation $P3$.

From relation $P_2$, $b^c = a^{(ba)^t}$.  This implies $b^{cbc} = a^{(ba)^t b c} = b$ (by $R_2$).
\end{proof}

Our new presentation for $Q_{(2,n)}(L_k)$ will be
$$Q_{(2,n)}(L_k) = \langle a, b, c \mid R_1, P_1, P_2, P_3, P_4, \{J_{xy} \mid x, y \in \{a,b,c\}\} \rangle.$$
Lemma \ref{L:relations} shows that all the relations in this presentation follow from $R_1$, $R_2$, $R_3$ and $J_{xy}$.  Conversely, $R_3$ is equivalent to $P_1$ by the second quandle axiom, and we derive $R_2$ as follows:
$$a^{(ba)^t b c} = a^{(ba)^t \bar{c} c b c} \stackrel{P2}{=} b^{cbc} \stackrel{P4}{=} b.$$
So the two presentations are equivalent.

We begin our construction of the Cayley graph by tracing the primary relations.  The graph will have two components; one contains the generators $a$ and $b$, and the other contains $c$.  We first trace out the component containing $a$ and $b$. We start at the vertex $a$, which we denote $x_{0,0}$, and add a loop labeled $a$.  We trace the relation $J_{ac}: a^{c^n} = a$, letting $x_{i,0}^c = x_{i+1,0}$ for $0 \leq i < n$ (where the first subscript is modulo $n$); this traces the innermost polygon in Figure \ref{F:Q(2,n)Lkodd}.  We then trace the relation $P_1: a^{\bar{c}(ba)^t} = b$.  Recall that $a^{\bar{c}} = x_{n-1,0}$.  Now we denote $a^{\bar{c}(ba)^i}$ by $x_{n-1,2i}$ and $a^{\bar{c}(ba)^i b}$ by $x_{n-1, 2i+1}$.  This means that $b = x_{n-1, 2t}$. We add a loop labeled $b$ at vertex $x_{n-1, 2t}$ (see Figure \ref{F:Q(2,n)Lkodd}).

We now trace relation $P_2: a^{(ba)^t\bar{c}} = b$.  Denote $a^{(ba)^i}$ by $x_{0, 2i}$ and $a^{(ba)^i b}$ by $x_{0,2i+1}$ for $0 \leq i < t$.  Then $x_{0,2t}^{\bar{c}} = b = x_{n-1,2t}$.  Next we trace the relation $J_{bc}: b^{c^n} = b$.  This gives us the outermost polygon in Figure \ref{F:Q(2,n)Lkodd}; we let $x_{i,2t}^c = x_{i+1, 2t}$ for $0 \leq i < n$, with the first subscript computed modulo $n$.  (This is consistent with $x_{0,2t}^{\bar{c}} = x_{n-1,2t}$.)  We now have the innermost and outermost polygons, and two paths connecting them.

The last primary relations for this component are $P_3$ and $P_4$. $P_3$ can be rewritten as $a^{ca} = a^{\bar{c}}$, which means $x_{1,0}^a = x_{n-1,0}$.  $P_4$ can be rewritten as $b^{cb} = b^{\bar{c}}$, so $x_{0,2t}^b = x_{n-2,2t}$.  These give one of the edges inside the inner polygon, and one of the edges outside the outer polygon (respectively) in Figure \ref{F:Q(2,n)Lkodd}.

To trace the remainder of the Cayley graph, we will use the secondary relations:
\begin{align*}
W_1&:\ x^{c^{ab}} = x^c \implies x^{bacab\bar{c}} = x \\
W_2&:\ x^{a^{\bar{c}(ba)^t}} = x^b \implies x^{(ab)^t ca\bar{c}(ba)^t b} = x \\
W_3&:\ x^{a^{(ba)^t\bar{c}}} = x^b \implies x^{c(ab)^ta(ba)^t \bar{c} b} = x \\
W_4&:\ x^{a^{cac}} = x^a \implies x^{(\bar{c}a)^2(ca)^2} = x \\
W_5&:\ x^{b^{cbc}} = x^b \implies x^{(\bar{c}b)^2(cb)^2} = x
\end{align*}

Note that $W_1$ is implied by Lemma \ref{L:r1}, part (3).  Using Lemma \ref{L:r1}, $W_2$ can be rewritten as 
$$x = x^{(ab)^t ca\bar{c}(ba)^t b} = x^{(ab)^t ca(ba)^t \bar{c} b} = x^{(ab)^t c (ab)^t a \bar{c} b} = x^{(ab)^{2t} ca\bar{c}b}.$$
And, similarly, $W_3$ can be rewritten as
$$x = x^{c(ab)^ta(ba)^t \bar{c} b} = x^{c(ab)^{2t} a \bar{c} b} = x^{(ab)^{2t} ca\bar{c} b}.$$
So $W_2$ and $W_3$ are equivalent relations, and we need only verify one of them at each vertex of the Cayley graph.

Lemma \ref{L:r1} says $x^{(ba)c} = x^{c(ba)}$ for any quandle element $x$, and hence (inductively) $x^{(ba)^i c} = x^{c(ba)^i}$ for any quandle element $x$ and positive integer $i$.  In particular, $a^{(ba)^{t} c} = a^{c(ba)^{t}} \implies x_{0,2t}^c = x_{1,0}^{(ba)^{t}} \implies x_{1,2t} = x_{1,0}^{(ba)^{t}}$. Proceeding inductively, we see $x_{i,2t} = x_{i,0}^{(ba)^{t}}$ for $0 \leq i < n$.  This gives us the radial paths in Figure \ref{F:Q(2,n)Lkodd}, and inspires us to denote $x_{i,2j} = x_{i,0}^{(ba)^j}$ and $x_{i, 2j+1} = x_{i,0}^{(ba)^j b}$.  This means $x_{i,2j-1}^a = x_{i,2j}$ and $x_{i,2j}^b = x_{i,2j+1}$. We claim that we have now labeled all the vertices in this component of the Cayley graph; it remains to show that all other edges are among these vertices, and that the graph will not collapse any further.

Our next lemma traces the edges inside the innermost polygon, and outside the outermost polygon.

\begin{lemma} \label{L:ends}
For every $i$ with $0 \leq i < n$, \begin{enumerate}
	\item $x_{i,0}^a = x_{n-i,0}$ and 
	\item $x_{i,2t}^b = x_{n-2-i,2t}$ 
\end{enumerate}
(where the first subscript is considered modulo $n$).
\end{lemma}
\begin{proof}
We will first prove part (1) by induction on $i$.  For our base case, since $x_{0,0} = a$, we know $x_{0,0}^a = x_{0,0}$.  Also, by relation $P_3$, $x_{0,0}^{cac} = x_{0,0} = x_{0,0}^a$, which implies $x_{0,0}^{(ca)^2} = x_{0,0}$.

For our inductive step, we assume both that $x_{i,0}^a = x_{n-i,0}$ (so $x_{n-i,0}^a = x_{i,0}$ as well) and $x_{i,0}^{(ca)^2} = x_{i,0}$. Then 
$$x_{i+1,0}^a = x_{i,0}^{ca} = x_{i,0}^{a\bar{c}} = x_{n-i,0}^{\bar{c}} = x_{n-i-1,0} = x_{n-(i+1),0}.$$
Relation $W_4$ implies $x^{(ca)^2} = x^{(ac)^2}$, so
$$x_{i+1,0}^{(ca)^2} = x_{i+1,0}^{(ac)^2} = x_{n-(i+1),0}^{cac} = x_{n-i,0}^{ac} = x_{i,0}^c = x_{i+1,0}.$$
This completes the inductive step, and the proof of part (1).

The proof of part (2) is similar, using relation $W_5$. Here, the base case is $x_{n-1,2t}$, and the inductive step moves from $x_{(n-1)+i,2t}$ to $x_{(n-1)+(i+1),2t}$.
\end{proof}

In particular, if $n$ is even, there is a loop labeled $a$ at $x_{n/2,0}$ and a loop labeled $b$ at $x_{(n/2)-1,2t}$.

Our next lemma traces the sides of the nested polygons.

\begin{lemma} \label{L:sides}
For every $i, j$ with $0 \leq i < n$ and $0 \leq j < t$, we have $x_{i, 2j}^c = x_{i+1,2j}$ and $x_{i,2j+1}^c = x_{i-1,2j+1}$.
\end{lemma}
\begin{proof}
The first part of the lemma comes directly from Lemma \ref{L:r1}.
$$x_{i,2j}^c = x_{i,0}^{(ba)^j c} = x_{i,0}^{c (ba)^j} = x_{i+1, 0}^{(ba)^j} = x_{i+1,2j}.$$
The second part uses both Lemma \ref{L:r1} and Lemma \ref{L:ends}.
\begin{align*}
x_{i,2j+1}^c &= x_{i,0}^{(ba)^j bc} = x_{n-i,0}^{a(ba)^jb c} = x_{n-i,0}^{(ab)^{j+1} c} = x_{n-i ,0}^{c (ab)^{j+1}} \\
&= x_{n-i+1,0}^{(ab)^{j+1}} = x_{n-(i-1),0}^{a(ba)^j b} = x_{i-1,0}^{(ba)^j b} = x_{i-1, 2j+1}.
\end{align*}
\end{proof}

We have now traced out all the edges in the components of the graphs in Figure \ref{F:Q(2,n)Lkodd} containing $a$ and $b$.  We still need to show that there is no further collapsing as we apply the secondary relations at each vertex. We do this by applying Lemmas \ref{L:ends} and \ref{L:sides}, along with the construction of the vertices $x_{i,j}$.

\begin{lemma} \label{L:secondary}
The Cayley graphs described above (and illustrated in Figure \ref{F:Q(2,n)Lkodd}) satisfy the secondary relations $W_1, W_2, W_3, W_4$ and $W_5$.
\end{lemma}
\begin{proof}
We consider each secondary relation in turn.  For $W_1$ we observe, for $0 \leq i < n$ and $0 \leq j \leq t-1$,
$$x_{i,2j}^{bacab\bar{c}} = x_{i,2j+2}^{cab\bar{c}} = x_{i+1,2j+2}^{ab\bar{c}} = x_{i+1,2j}^{\bar{c}} = x_{i,2j}$$
and for $1 \leq j \leq t-1$
$$x_{i,2j+1}^{bacab\bar{c}} = x_{i,2j-1}^{cab\bar{c}} = x_{i-1,2j-1}^{ab\bar{c}} = x_{i-1,2j+1}^{\bar{c}} = x_{i,2j+1}$$
We have two special cases, when $2j = 2t$ and $2j+1 = 1$, which use Lemma \ref{L:ends}:
\begin{align*}
x_{i,2t}^{bacab\bar{c}} &= x_{(n-2)-i,2t}^{acab\bar{c}} = x_{(n-2)-i,2t-1}^{cab\bar{c}} = x_{(n-3)-i,2t-1}^{ab\bar{c}} \\
&= x_{(n-3)-i,2t}^{b\bar{c}} = x_{i+1,2t}^{\bar{c}} = x_{i,2t}.\\
x_{i,1}^{bacab\bar{c}} &= x_{i,0}^{acab\bar{c}} = x_{n-i,0}^{cab\bar{c}} = x_{n-i+1,0}^{ab\bar{c}} = x_{i-1,0}^{b\bar{c}} = x_{i-1,1}^{\bar{c}} = x_{i,1}.
\end{align*}

\medskip

Recall that relations $W_2$ and $W_3$ are both equivalent to $x^{(ab)^{2t} ca\bar{c}b} = x$; we will verify this version of the relation.  Suppose $0 \leq i < n$ and $0 \leq j \leq t$. (Note there is a special case for $x_{i,1}$, to avoid negative subscripts; and in the second case, $j \leq t-1$.)
\begin{align*}
x_{i,2j}^{(ab)^{2t} ca\bar{c}b} &= x_{i,0}^{(ab)^{2t-j}ca\bar{c}b} = x_{n-i,0}^{(ba)^{2t-j}a ca\bar{c}b} = x_{n-i,2t}^{(ba)^{t-j}aca\bar{c}b} = x_{i-2,2t}^{(ab)^{(t-j)}baca\bar{c}b} \\
&= x_{i-2,2j}^{baca\bar{c}b} = x_{i-2,2j+2}^{ca\bar{c}b} = x_{i-1,2j+2}^{a\bar{c}b} = x_{i-1,2j+1}^{\bar{c}b} = x_{i,2j+1}^b = x_{i,2j}. \\
x_{i,2j+1}^{(ab)^{2t} ca\bar{c}b} &= x_{i,2t-1}^{(ab)^{t+j+1}ca\bar{c}b} = x_{i,2t}^{(ba)^{t+j}bca\bar{c}b} = x_{(n-2)-i,2t}^{(ab)^{t+j} ca\bar{c}b} = x_{(n-2)-i,0}^{(ab)^{j} ca\bar{c}b} = x_{i+2,0}^{(ba)^{j} aca\bar{c}b} \\
&= x_{i+2,2j}^{aca\bar{c}b} = x_{i+2,2j-1}^{ca\bar{c}b} = x_{i+1,2j-1}^{a\bar{c}b} = x_{i+1,2j}^{\bar{c}b} = x_{i,2j}^b = x_{i,2j+1}. \\ 
x_{i,1}^{(ab)^{2t} ca\bar{c}b} &= x_{i,2t-1}^{(ab)^{t+1}ca\bar{c}b} = x_{i,2t}^{(ba)^{t}bca\bar{c}b} = x_{(n-2)-i,2t}^{(ab)^{t} ca\bar{c}b} = x_{(n-2)-i,0}^{ca\bar{c}b} \\
&= x_{(n-1)-i,0}^{a\bar{c}b} = x_{i+1,0}^{\bar{c}b} = x_{i,0}^b = x_{i,1}. 
\end{align*}

\medskip

To verify relations $W_4$ and $W_5$, we actually show that $x^{(ca)^2} = x^{(\bar{c}a)^2} = x^{(cb)^2} = x^{(\bar{c}b)^2} = x$ for all $x$.  This immediately implies relations $W_4$ and $W_5$ hold.  We will show $x^{(ca)^2} = x$; the other cases are similar. Suppose $0 \leq i < n$ and $0 \leq j \leq t$. (Note there is a special case for $x_{i,0}$, to avoid negative subscripts; and in the second case, $j \leq t-1$.)
\begin{align*}
x_{i,2j}^{caca} &= x_{i+1,2j}^{aca} = x_{i+1,2j-1}^{ca} = x_{i,2j-1}^a = x_{i,2j}. \\
x_{i,2j+1}^{caca} &= x_{i-1,2j+1}^{aca} = x_{i-1,2j+2}^{ca} = x_{i,2j+2}^a = x_{i,2j+1}. \\
x_{i,0}^{caca} &= x_{i+1,0}^{aca} = x_{n-i-1,0}^{ca} = x_{n-i,0}^a = x_{i,0}. 
\end{align*}

This verifies that all the secondary relations hold for these graphs.
\end{proof}

Finally, we turn to the component containing $c$.  We start with the vertex for $c$, with a loop labeled $c$.  We have a second vertex at $c^a$, and by primary relation $R_1$ we know $c^a = c^b$. Then we observe that, by relation $W_2$, 
$$c^{(ab)^{2t} ca\bar{c}b} = c \implies c^{ca\bar{c}b} = c \implies c^{a\bar{c}} = c^b \implies (c^a)^{\bar{c}} = c^a \implies (c^a)^c = c^a.$$
So there is also a loop labeled $c$ at $c^a$.  It is now easy to check that all the secondary relations are satisfied, so the second component is just these two vertices.

This proves part (1) of Theorem \ref{T:QNLk}.

\subsection{$L_k$ with $k$ even, $k \neq 0$} \label{SSS:even}

We now turn to the case when $k = 2t$ is even (and $k \neq 0$). In this case the link has three components, and the quandle $Q_{(2,2,n)}(L_k)$ has the following presentation (this can be seen through an easy inductive argument), where $x$ is any of the generators $a, b, c$:
$$Q_{(2,2,n)}(L_k) = \langle a, b, c \mid c^{ab} = c,\ a^{(ba)^{t-1} bc} = a,\ b^{(ab)^t c} = b, x^{a^2} = x^{b^2} = x^{c^n} = x \rangle.$$
Once again, we will denote the relations $R_1$, $R_2$, $R_3$ and $J_{xy}$ (where $J_{xy}$ is the relation $x^{y^{n_y}} = x$, for $x, y \in \{a, b, c\}$). Also, as with $k$ odd (see Remark \ref{R:k<0}), it is easy to see that $Q_{(2,2,n)}(L_{-k})$ is isomorphic to $Q_{(2,2,n)}(L_k)$, so we may assume $k > 0$. As in Section \ref{SSS:odd}, we will construct the Cayley graph for $Q_{(2,2,n)}(L_k)$; the graphs look slightly different for $n$ odd and $n$ even.  Figures \ref{F:Q(2,2,n)Lkeven5} and \ref{F:Q(2,2,n)Lkeven6} show the Cayley graphs for $n = 5$ and $n = 6$, respectively.  As we can see, they look very similar to the graphs when $k$ is odd shown in Figure \ref{F:Q(2,n)Lkodd}, except that the large component has been split into two isomorphic pieces.

\begin{figure}[htbp]
$$\includegraphics[height=7.5in]{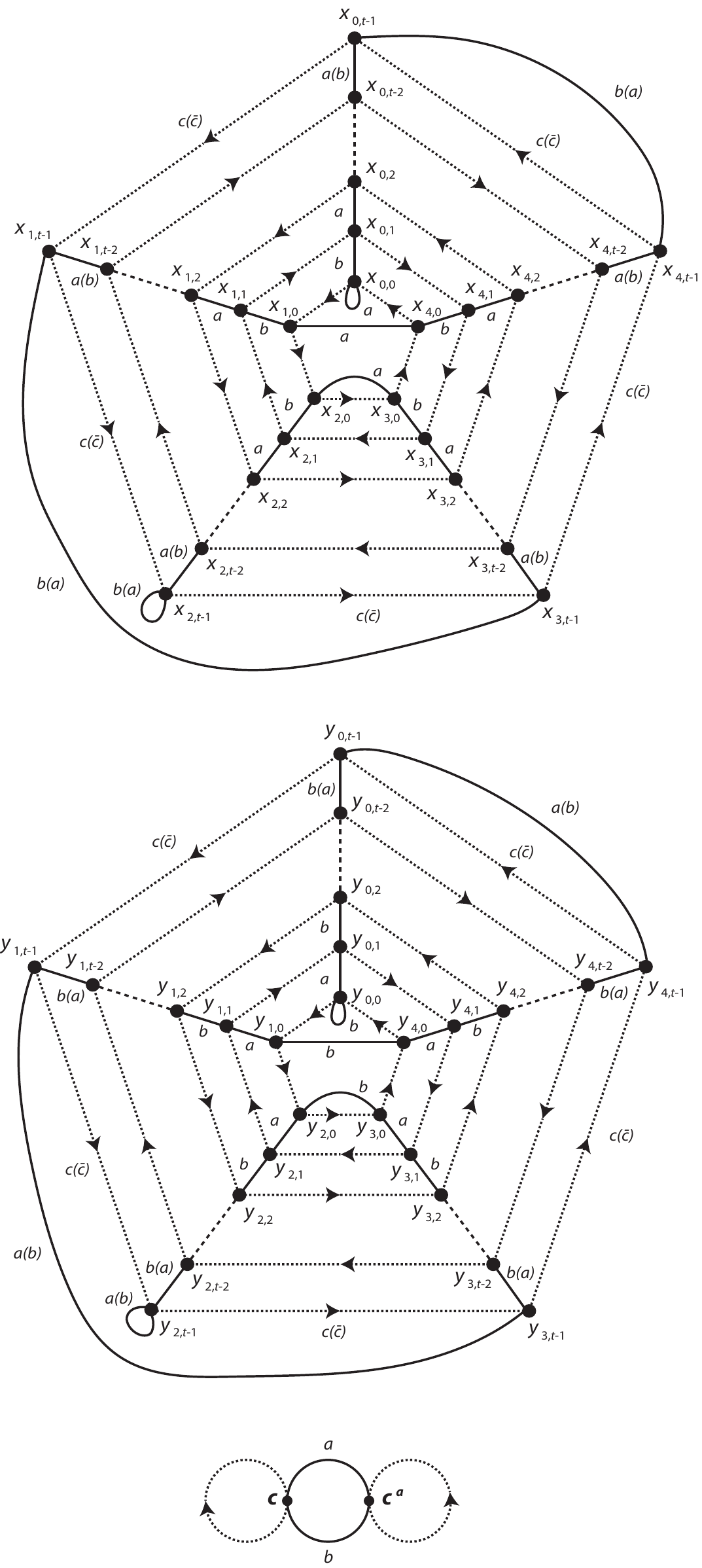}$$
\caption{$Q_{(2,5)}(L_k)$ for $k = 2t$; dotted edges represent generator $c$. Labels such as $a(b)$ represent $a$ when $t$ is even and $b$ when $t$ is odd.}
\label{F:Q(2,2,n)Lkeven5}
\end{figure}

\begin{figure}[htbp]
$$\includegraphics[height=7.5in]{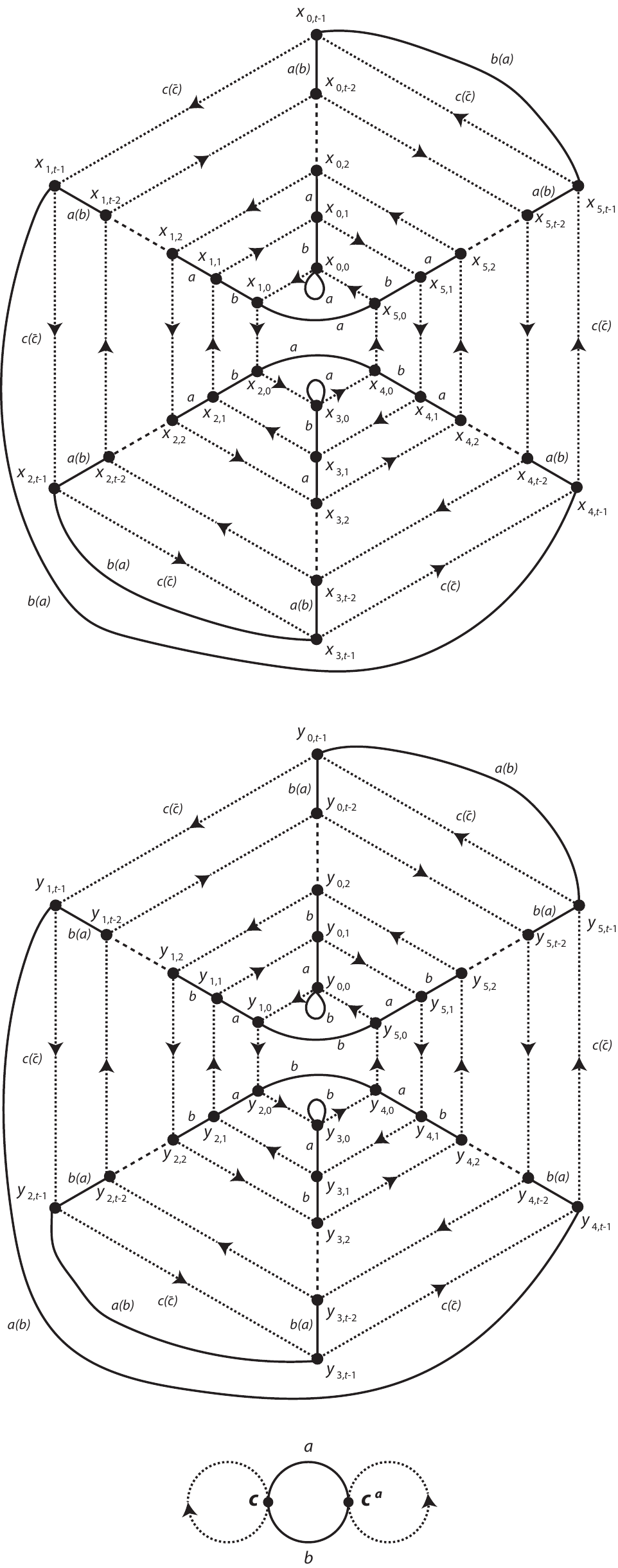}$$
\caption{$Q_{(2,6)}(L_k)$ for $k = 2t$; dotted edges represent generator $c$. Labels such as $a(b)$ represent $a$ when $t$ is even and $b$ when $t$ is odd.}
\label{F:Q(2,2,n)Lkeven6}
\end{figure}

We are going to derive another presentation for $Q_{(2,2,n)}(L_k)$ that will be easier to use with Winker's algorithm.  Note that, since relation $R_1$ is the same as in Section \ref{SSS:odd}, the relations in Lemma \ref{L:r1} still hold. The following relations from Lemma \ref{L:relations} also hold when $k$ is even (to avoid confusion, we will give them the same labels we did in the previous section).

\begin{lemma} \label{L:relations2}
The following relations hold in $Q_{(2,2,n)}(L_k)$ ($k$ even).
\begin{enumerate}
	\item $P_3: a^{cac} = a$
	\item $P_4: b^{cbc} = b$
\end{enumerate}
\end{lemma}
\begin{proof}
From relation $R_2$, $a = a^{(ba)^{t-1} b c} = a^{(ab)^t c} = a^{c(ab)^{t}}$ (by Lemma \ref{L:r1}).  Hence, $a^c = a^{(ba)^{t}} = a^{(ba)^{t-1} b (c \bar{c}) a} = a^{\bar{c} a}$ by relation $R_2$.  This implies $a^{cac} = a$, giving relation $P3$.

From relation $R_3$, $b = b^{(ab)^t c} = b^{c(ab)^t}$.  Hence $b^c = b^{(ba)^t} = b^{(ab)^t b}$, which means $b^{cbc} = b^{(ab)^t c} = b$, proving relation $P_4$.
\end{proof}

So we can give a new presentation for $Q_{(2,2,n)}(L_k)$:
$$Q_{(2,2,n)}(L_k) = \langle a, b, c \mid R_1, R_2, R_3, P_3, P_4, \{J_{xy} \mid x, y \in \{a,b,c\}\} \rangle.$$
Since this contains all the relations of our original presentation, and by Lemma \ref{L:relations2} the new relations are consequences of the old ones, this presentation is equivalent to the original one.

We begin by tracing the primary relations.  We first look at the component of the graph containing generator $a$. For this component, the only primary relations we need to consider are $R_2, P_3, J_{ab}$ and $J_{ac}$ ($J_{aa}$ is an immediate consequence of the quandle axioms).  We begin with vertex $a$, which we denote $x_{0,0}$, and the loop at $x_{0,0}$ labeled $a$. Relation $J_{ac}: a^{c^n} = a$ traces out the innermost polygon in Figures \ref{F:Q(2,2,n)Lkeven5} and \ref{F:Q(2,2,n)Lkeven6}; we denote these vertices $x_{i,0}$, with $x_{i,0}^c = x_{i+1,0}$, and the first subscript taken modulo $n$.  Now, for $0 \leq j$, we let $x_{i,2j}^b = x_{i,2j+1}$ and $x_{i,2j+1}^a = x_{i,2j+2}$, until the second subscript is $t-1$; these will be the radial segments of the Cayley graphs.  We will ultimately see that these are all the vertices in this component of the Cayley graph.

We now trace out relation $R_2: a^{(ba)^{t-1} bc} = a$.  If $t$ is odd, we can rewrite this as $a^{(ba)^{(t-1)/2}b} = a^{\bar{c}(ba)^{(t-1)/2}}$, which means $x_{0,t-1}^b = x_{n-1,t-1}$.  If $t$ is even, we have $a^{(ba)^{(t-2)/2}ba} = a^{\bar{c}(ba)^{(t-2)/2}b}$, which means $x_{0,t-1}^a = x_{n-1,t-1}$.  This gives the edge connecting $x_{0,t-1}$ to $x_{n-1,t-1}$ in Figures \ref{F:Q(2,2,n)Lkeven5} and \ref{F:Q(2,2,n)Lkeven6}; the edge is labeled $b$ if $t$ is odd and $a$ if $t$ is even.

Finally, we trace out relation $P_3: a^{cac} = a$.  We can rewrite this as $a^{ca} = a^{\bar{c}}$, or $x_{1,0}^a = x_{n-1,0}$, which gives the edge labeled $a$ between $x_{1,0}$ and $x_{n-1,0}$.

Now we turn to the secondary relations:
\begin{align*}
W_1&:\ x^{c^{ab}} = x^c \implies x^{bacab\bar{c}} = x \\
W_2&:\ x^{a^{(ba)^{t-1} b c}} = x^a \implies x^{\bar{c} (ba)^{2t-1} b c a} = x \\
W_3&:\ x^{b^{(ab)^tc}} = x^b \implies x^{\bar{c}(ba)^{2t} b c b} = x \\
W_4&:\ x^{a^{cac}} = x^a \implies x^{(\bar{c}a)^2(ca)^2} = x \\
W_5&:\ x^{b^{cbc}} = x^b \implies x^{(\bar{c}b)^2(cb)^2} = x
\end{align*}
Observe that $W_2$ can be rewritten as $x^{\bar{c} (ba)^{2t} a c a} = x$.  By Lemma \ref{L:r1}, $x^{aca} = x^{bcb}$ for any $x$ in the quandle, so $W_2$ is equivalent to $W_3$, and we need only check one of them.

We now prove a lemma that traces the edges inside the inner polygon.

\begin{lemma} \label{L:ends2}
For every $i$ with $0 \leq i < n$, $x_{i,0}^a = x_{n-i,0}$ (where the first subscript is considered modulo $n$).
\end{lemma}
\begin{proof}
The proof is identical to the proof of part (1) of Lemma \ref{L:ends}.
\end{proof}

We use this lemma, along with Lemma \ref{L:r1}, to trace out the sides of the nested polygons.  The proof is identical to Lemma \ref{L:ends}, using Lemma \ref{L:ends2} in place of Lemma \ref{L:ends}.

\begin{lemma} \label{L:sides2}
For every $i, j$ with $0 \leq i < n$ and $0 \leq j$, we have $x_{i, 2j}^c = x_{i+1,2j}$ and $x_{i,2j+1}^c = x_{i-1,2j+1}$ (where the second subscript is at most $t-1$).
\end{lemma}

Finally, we trace the edges outside the outer polygon.

\begin{lemma} \label{L:ends3}
For every $i$ with $0 \leq i < n$, $x_{i,t-1}^{a(b)} = x_{n-1-i,t-1}$, where $a(b)$ represents $a$ if $t$ is even and $b$ if $t$ is odd (and where the first subscript is considered modulo $n$).
\end{lemma}
\begin{proof}
We first consider the case when $t$ is even.  We proceed by induction on $i$; we already know that $x_{0,t-1}^a = x_{n-1,t-1}$. 

Now assume that for $0 \leq i < h \leq n/2$, $x_{i,t-1}^a = x_{n-1-i,0}$. From relation $W_4$, $x_{h-1,t-1}^{(\bar{c}a)^2(ca)^2} = x_{h-1,t-1}$.  Using Lemma \ref{L:sides2} and the inductive hypothesis, this implies
\begin{align*}
x_{h,t-1}^a &= x_{h-1,t-1}^{(a\bar{c})^2 ac} = x_{(n-1)-(h-1),t-1}^{\bar{c}a\bar{c}ac} \\
&= x_{(n-1)-(h-2),t-1}^{a\bar{c}ac} = x_{h-2,t-1}^{\bar{c}ac} \\
&= x_{h-1,t-1}^{ac} = x_{(n-1)-(h-1),t-1}^c = x_{n-1-h,t-1}.
\end{align*}
Hence, by induction, $x_{i,t-1}^{a} = x_{n-1-i,t-1}$ for $0 \leq i \leq (n-1)/2$. Since this means $x_{n-1-i,t-1}^a = x_{i,t-1}$, we also get the result for $(n-1)/2 \leq i \leq n-1$.

If $t$ is odd, the proof is almost identical, replacing $a$ with $b$.  We use relation $W_5$ instead of $W_4$, and in the inductive step we start at $x_{(n-1)-(h-1),t-1}$ instead of at $x_{h-1,t-1}$.  We conclude that $x_{(n-1)-h,t-1}^b = x_{h,t-1}$, and hence $x_{h,t-1}^b = x_{(n-1)-h,t-1}$. 
\end{proof}

We have now completely traced out the component of the Cayley graph containing $a$ (as shown in Figures \ref{F:Q(2,2,n)Lkeven5} and \ref{F:Q(2,2,n)Lkeven6}).  It remains to show there is no further collapsing by showing that the secondary relations are satisfied at every vertex.  The proof of this is very similar to Lemma \ref{L:secondary}, and is left to the reader.

The component containing $b$ is determined in almost exactly the same way, interchanging $a$ and $b$, and the component containing $c$ is easily determined as in the case when $k$ is odd.  We have completed our construction of the Cayley graph for $Q_{(2,2,n)}(L_k)$ when $k$ is even, and as a corollary we have proved part (2) of Theorem \ref{T:QNLk}.

\section{The family of links $M_k = T_k \cup C$} \label{SS:Mk}

Our last family of links, $M_k$, consists of a twist knot $T_k$ linked with another component $C$, as shown in Figure \ref{F:Mk}.  This link has two components; we will derive a presentation for the quandle $Q_{(2,3)}(M_k)$ with three generators $a$, $b$ and $c$, shown in Figure \ref{F:Mk}.

\begin{figure}[htbp]
$$\includegraphics[height=1.5in]{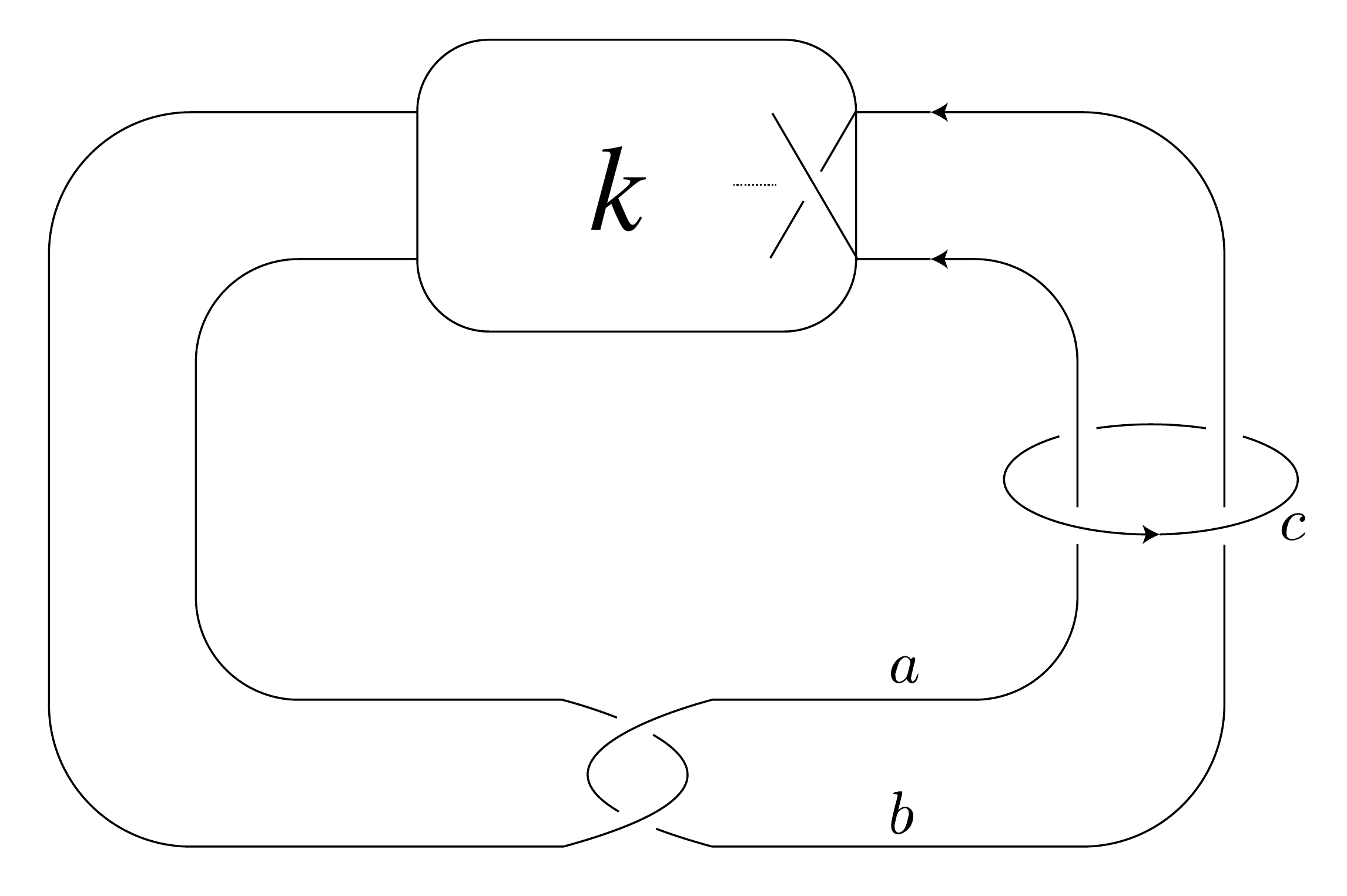}$$
\caption{$M_k = T_k \cup C$, where $T_k$ is a twist knot.}
\label{F:Mk}
\end{figure}

In the quandle $Q_{(2,3)}(M_k)$, we assume that, for any $x$ in the quandle, $x^{a^2} = x^{b^2} = x^{c^3} = x$.  We will use those relations as we derive our presentation.

The linking of the component $C$ with the twist knot gives the following relation:
$$c^{b^ca^c} = c \implies c^{\bar{c}bc\bar{c}ac} = c \implies c^{ba} = c^{\bar{c}} = c.$$
This is the same as relation $R_1$ in Section \ref{SS:Lk}, so the relations in Lemma \ref{L:r1} still hold. We now turn to the relations induced by the twists along the top of the diagram, and the clasp at the bottom.

The following lemma (proved in \cite{Me}) is the result of an easy inductive argument.

\begin{lemma} \label{L:halftwists}
The arcs on either side of the block of $k$ right-handed half-twists are labeled as shown below (for $k$ even and $k$ odd).  Here $X = (ba)^tc$ and $Y = (ba)X = (ba)^{t+1}c$. (If $k < 0$, there are $\vert k \vert$ left-handed half-twists; the same formulas hold, where $(ba)^{-1} = \overline{ba} = ab$.)
$$\scalebox{1}{\includegraphics{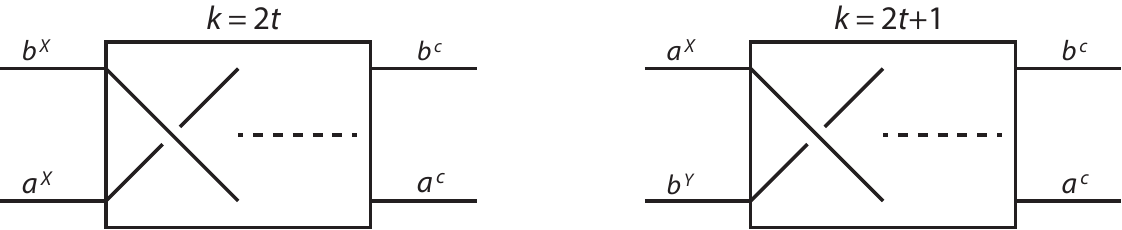}}$$
\end{lemma}

When $k = 2t$ is even, the clasp at the bottom of the diagram in Figure \ref{F:Mk} gives the two relations:
$$a^{(ba)^t c a} = b^{(ba)^t c} \qquad {\rm and} \qquad a^{b^{(ba)^t c}} = b.$$
We can rewrite the first relation as follows, using Lemma \ref{L:r1}:
\begin{align*}
a^{(ba)^t c a} = b^{(ba)^t c} &\implies a^{c a(ab)^t} = b^{(ba)^t c} \\
&\implies a^{ca} = b^{(ba)^t c (ba)^t} = b^{(ba)^{2t} c} \\
&\implies a^{ca\bar{c}} = b^{(ba)^k} = b^{(ab)^{k-1} a} \\
&\implies a^{ca\bar{c}a} = b^{(ab)^{k-1}} 
\end{align*}
And, similarly, we can rewrite the second relation as:
\begin{align*}
a^{b^{(ba)^t c}} = b &\implies a^{\bar{c}(ab)^t b (ba)^t c} = b \\
&\implies a^{\bar{c} (ab)^{2t-1} a c} = b \\
&\implies a^{\bar{c} a c (ba)^{2t-1}} = b \\
&\implies a^{\bar{c} a c} = b^{(ab)^{2t-1}} = b^{(ab)^{k-1}}
\end{align*}
Combining these gives us $a^{ca\bar{c}a} = a^{\bar{c}ac}$, or $a^{ca(\bar{c}a)^2 c} = a$.

When $k = 2t+1$ is odd, the clasp gives us relations:
$$b^{(ba)^{t+1} c a} = a^{(ba)^t c} \qquad {\rm and} \qquad a^{a^{(ba)^t c}} = b.$$
However, a similar argument to the case when $k$ is even once again gives the relations $a^{ca\bar{c}a} = b^{(ab)^{k-1}} = a^{\bar{c} a c}$.  So in either case, we get the following presentation for $Q_{(2,3)}(M_k)$, where $x$ is any of the generators $a, b, c$:
$$Q_{(2,3)}(M_k) = \langle a, b, c \mid c^{ba} = c,\ a^{ca\bar{c}a} = a^{\bar{c}ac},\ a^{\bar{c}ac} = b^{(ab)^{k-1}},\ x^{a^2} = x^{b^2} = x^{c^3} = x  \rangle$$
We will denote the relations $R_1$, $R_2$, $R_3$, $J_{xy}$ (where $J_{xy}$ is the relation $x^{y^{n_y}} = x$ for $x, y \in \{a, b, c\}$).

\begin{remark} \label{R:Mk,k<0}
If $k \leq 0$, then relation $R_3$ becomes 
$$a^{\bar{c}ac} = b^{(ab)^{k-1}} = b^{(ba)^{-k+1}} = b^{(ab)^{-k}a} \implies a^{\bar{c}aca} = b^{(ab)^{-k}}$$
and we can rewrite $R_2$ as
$$a^{ca\bar{c}} = a^{\bar{c}aca} \implies a^{ca\bar{c}} = b^{(ab)^{-k}}.$$
Since $-k = (1-k)-1 = (1+\vert k \vert )-1$, we see that the relations for $Q_{(2,3)}(M_k)$ are the same as for $Q_{(2,3)}(M_{\vert k\vert+1})$, except that $c$ and $\bar{c}$ are reversed.  But this is an isomorphism, so it suffices to consider the case when $k > 0$.  Moreover, when $k < 0$, $\vert 2k-1\vert = -(2k-1) = 2\vert k\vert +1 = 2(\vert k \vert +1)-1$, so the formula in Theorem \ref{T:QNMk} will still hold.
\end{remark}

\begin{figure}[htbp]
$$\includegraphics[width=4.75in]{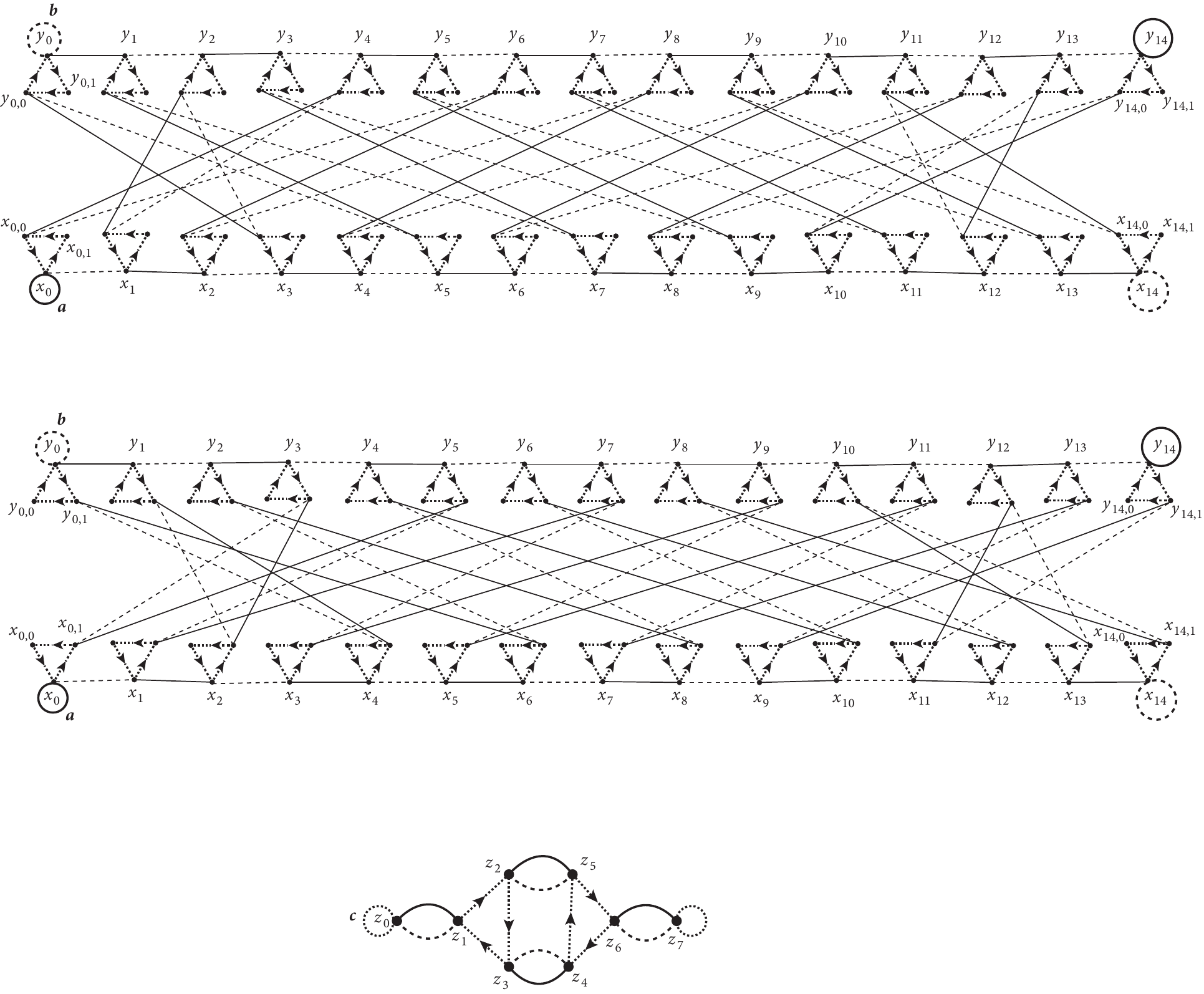}$$
\caption{Cayley graph for $Q_{(2,3)}(M_3)$. Edges corresponding to the operation of edges $a$, $b$ and $c$ are represented by solid, dashed and dotted lines, respectively.  For clarity, the component containing $a$ and $b$ has been split into two pieces.}
\label{F:k=3}
\end{figure}

We will now proceed to construct the Cayley graph for this quandle.  An example for $k=3$ is shown in Figure~\ref{F:k=3}. Although the Cayley graph has two components, we have drawn the larger component in two pieces; the first shows the edges labeled $a$ and $b$ adjacent to the vertices $x_{i,0}$ and the other shows the edges labeled $a$ and $b$ adjacent to the vertices $x_{i,1}$ (we will explain this labeling when we trace out the primary relations). As a corollary of our construction, we will show that the quandle is finite by proving Theorem \ref{T:QNMk}.

We will now prove some relations in $Q_{(2,3)}(M_k)$ that will help us trace out the Cayley graph.  We first rewrite relations $R_2$ and $R_3$.

\begin{lemma} \label{L:relationsb}
\begin{enumerate}
	\item $b^{aca\bar{c}} = a^{(ba)^{k-1}}$
	\item $b^{\bar{c}aca} = a^{(ba)^{k-1}}$
\end{enumerate}
\end{lemma}
\begin{proof}
\begin{align*}
b^{(ab)^{k-1}} = a^{ca\bar{c}a} &\implies b^{(ab)^{k-1}aca\bar{c}} = a \\
&\implies b^{aca\bar{c}(ab)^{k-1}} = a \ {\rm by\ Lemma\ }\ref{L:r1} \\
&\implies b^{aca\bar{c}} = a^{(ba)^{k-1}}
\end{align*}
\begin{align*}
b^{(ab)^{k-1}} = a^{\bar{c}ac} &\implies b^{(ab)^{k-1}\bar{c}ac} = a \\
&\implies b^{\bar{c}a(ba)^{k-1}c} = b^{\bar{c}ac(ba)^{k-1}} = a \\
&\implies b^{\bar{c}ac} = a^{(ab)^{k-1}} = a^{(ba)^{k-1}a} \\
&\implies b^{\bar{c}aca} = a^{(ba)^{k-1}}
\end{align*}
\end{proof}

\begin{lemma} \label{L:Mksides}
For any integer $k$, the following relations hold in $Q_{(2,3)}(M_k)$. \begin{enumerate}
	\item $P_1: a^{(ba)^{3(2k-1)}} = a$
	\item $P_2: b^{(ab)^{3(2k-1)}} = b$
\end{enumerate}
\end{lemma}
\begin{proof}
We are making several uses of Lemma \ref{L:r1}, along with the fact that $a^{\bar{c}aca} = a^{ca\bar{c}}$ (from relation $R_2$) and the fact that $x^{c^3} = x$.
\begin{align*}
a^{(ba)^{3(2k-1)}} &= a^{(ba)^{6k-3}} = a^{(ba)^{k-1}(ba)^{5k-2}} \\
&= b^{aca\bar{c}(ba)^{5k-2}} = b^{(ba)^{5k-2}aca\bar{c}} = b^{(ab)^{5k-3}ca\bar{c}} \\
&= a^{\bar{c}ac(ab)^{4k-2}ca\bar{c}} = a^{(ba)^{4k-2}\bar{c}acca\bar{c}} = a^{(ba)^{4k-2}\bar{c} a \bar{c} a \bar{c}} \\
&= b^{aca\bar{c} (ba)^{3k-1} \bar{c} a \bar{c} a \bar{c}} = b^{(ba)^{3k-1} aca\bar{c} \bar{c} a \bar{c} a \bar{c}} = b^{(ab)^{3k-2} caca\bar{c} a \bar{c}} \\
&= a^{\bar{c}ac (ab)^{2k-1} caca\bar{c} a \bar{c}} = a^{(ba)^{2k-1} \bar{c}a \bar{c} aca\bar{c} a \bar{c}} \\
&= b^{aca\bar{c} (ba)^k \bar{c}a \bar{c} aca\bar{c} a \bar{c}} = b^{(ba)^k acac a \bar{c} aca\bar{c} a \bar{c}} \\
&= b^{(ab)^{k-1} cac a \bar{c} aca\bar{c} a \bar{c}} = a^{\bar{c}ac cac a \bar{c} aca\bar{c} a \bar{c}} = a^{\bar{c}(a \bar{c} ac a \bar{c} a ca)\bar{c} a \bar{c}} \\
&= a^{\bar{c}a^{\bar{c}aca} \bar{c} a \bar{c}} = a^{\bar{c}a^{ca\bar{c}} \bar{c} a \bar{c}} = a^{\bar{c} (ca\bar{c} a ca \bar{c}) \bar{c} a \bar{c}} \\
&= a^{(\bar{c} a c a) c a \bar{c}} = a^{ca\bar{c} ca\bar{c}} = a
\end{align*}
Relation $P_2$ can be proved similarly, reversing the roles of $a$ and $b$.  Notice that if we combine Lemma \ref{L:relationsb} with $x^{a\tilde{c}a} = x^{b\tilde{c}b}$ (from Lemma \ref{L:r1}), we get $b^{cb\bar{c}} = b^{\bar{c}bcb} = a^{(ba)^{k-1}}$. Similarly, $a^{bcb\bar{c}} = a^{\bar{c}bcb} = b^{(ab)^{k-1}}$.  So all the relations we need hold if we reverse $a$ and $b$.
\end{proof}

We can now give another presentation of $Q_{(2,3)}(M_k)$ which is helpful for constructing the Cayley graph.
$$Q_{(2,3)}(M_k) = \langle a, b, c \mid R_1, R_2, R_3, P_1, P_2 \ x^{a^2} = x^{b^2} = x^{c^3} = x  \rangle$$
Since we've only added the relations $P_1$ and $P_2$, which we've already proved are conseequences of the others, this presentation is equivalent to the original one.  

Now we will begin our Cayley graph by tracing out the primary relations in this presentation. We will first consider the component of the graph containing generators $a$ and $b$.  We will denote the vertex $a$ by $x_0$ and the vertex $b$ by $y_0$ (with loops labeled $a$ and $b$, respectively).  Then we denote $a^{(ba)^i}$ by $x_{2i}$ and $a^{(ba)^i b}$ by $x_{2i+1}$.  This means $x_{2i}^b = x_{2i+1}$ and $x_{2i+1}^a = x_{2i+2}$. Similarly, we denote $b^{(ab)^i}$ by $y_{2i}$ and $b^{(ab)^i a}$ by $y_{2i+1}$ (so $y_{2i}^a = y_{2i+1}$ and $y_{2i+1}^a = y_{2i}$). Relations $P_1$ and $P_2$ tell us that $x_{6(2k-1)} = x_0$ and $y_{6(2k-1)} = y_0$, so the subscripts can be read modulo $6(2k-1)$.  In fact, if we trace out relation $P_1$ a bit more carefully, we find
\begin{align*}
a^{(ba)^{3(2k-1)}} = a &\implies a^{(ba)^{6k-3}} = a \\
&\implies a^{(ba)^{3k-2}} = a^{(ab)^{3k-1}} = a^{(ba)^{3k-2} b} \\
&\implies x_{6k-4}^b = x_{6k-4}.
\end{align*}
Similarly, tracing out relation $P_2$ implies $y_{6k-4}^a = y_{6k-4}$.  So relations $P_1$ and $P_2$ trace out the bottom and top (respectively) of the first component in Figure \ref{F:k=3}, with the loops on the right-hand side.  More generally, we observe that 
$$x_{2i} = a^{(ba)^{i}} = a^{(ab)^{(6k-3)-i}} = a^{(ba)^{6k-4-i}b} = x_{12k-7-2i}$$
and
$$x_{2i+1} = a^{(ba)^{i}b} = a^{(ba)^{6k-4-i}} = x_{12k-8-2i} = x_{12k-7-(2i+1)}.$$
So for any $i$, $x_i = x_{12k-7-i} = x_{12k-6+i}$. Since $12k-7-i = 12k-6 + (-1-i)$, we also see $x_i = x_{-1-i}$.  Then any $x_i$ is equivalent to one with $0 \leq i \leq 6k-4$ (and similarly for $y_i$).

Now, at each vertex $x_i$ and $y_i$, we trace the relations $x_i^{c^3} = x_i$ and $y_i^{c^3} = y_i$.  We will denote $x_i^c$ by $x_{i,1}$ and $x_i^{c^2} = x_i^{\bar{c}}$ by $x_{i,0}$, and similarly for $y_{i,1}$ and $y_{i,0}$, as in Figure \ref{F:k=3}. We will see that these are all the vertices in this component of the graph. It remains to trace the edges labeled $a$ and $b$ that connect the vertices $x_{i,j}$ to the vertices $y_{i,j}$.

We now trace relation $R_3: a^{\bar{c} a c} = b^{(ab)^{k-1}}$.  We can rewrite this as $a^{\bar{c} a} = b^{(ab)^{k-1} \bar{c}}$, which means $x_{0,0}^a = y_{2k-2,0}$.  Relation $R_2$ tells us that $a^{ca\bar{c}a} = a^{\bar{c} a c} = b^{(ab)^{k-1}}$, so $a^{ca} = b^{(ab)^{k-1}ac}$.  This means $x_{0,1}^a = y_{2k-1,1}$.

The following lemma traces out the remaining edges, using relations $R_1$, $R_2$ and $R_3$ (and their consequences).

\begin{lemma} \label{L:struts}
For any integer $i$, we have \begin{enumerate}
\begin{multicols}{2}
	\item $x_{2i,0}^a = y_{2k+2i-2,0}$
	\item $x_{2i,0}^b = y_{2k+2i,0}$
	\item $x_{2i,1}^a = y_{2k-2i-1,1}$
	\item $x_{2i,1}^b = y_{2k-2i-3,1}$
	\item $x_{2i+1,0}^a = y_{2k-2i-4,0}$
	\item $x_{2i+1,0}^b = y_{2k-2i-2,0}$
	\item $x_{2i+1,1}^a = y_{2k+2i+1,1}$
	\item $x_{2i+1,1}^b = y_{2k+2i-1,1}$ 
	\item $y_{2i,0}^a = x_{2k-2i-3,0}$
	\item $y_{2i,0}^b = x_{2k-2i-1,0}$
	\item $y_{2i,1}^a = x_{2k+2i,1}$
	\item $y_{2i,1}^b = x_{2k+2i-2,1}$
	\item $y_{2i+1,0}^a = x_{2k+2i-1,0}$
	\item $y_{2i+1,0}^b = x_{2k+2i+1,0}$
	\item $y_{2i+1,1}^a = x_{2k-2i-2,1}$
	\item $y_{2i+1,1}^b = x_{2k-2i-4,1}$
\end{multicols}
\end{enumerate}
\end{lemma}
\begin{proof}
To prove part (1) and (2), observe
$$x_{2i,0}^a = a^{(ba)^i \bar{c} a} = a^{\bar{c} a (ab)^i} = a^{(\bar{c} a c) \bar{c} (ab)^i} \stackrel{R_3}{=} b^{(ab)^{k-1} \bar{c} (ab)^i} = b^{(ab)^{k-1 + i} \bar{c}} = y_{2k+2i-2,0}$$
and
$$x_{2i,0}^b = a^{(ba)^i \bar{c} b} = a^{\bar{c} b (ab)^i} = a^{(\bar{c} a c) \bar{c} (ab)^{i+1}} \stackrel{R_3}{=} b^{(ab)^{k-1} \bar{c} (ab)^{i+1}} = b^{(ab)^{k + i} \bar{c}} = y_{2k+2i,0}.$$
Parts (3) and (4) are proved similarly, using the relation $a^{ca\bar{c}a} = b^{(ab)^{k-1}}$ (from combining $R_2$ and $R_3$).
\begin{align*}
x_{2i,1}^a = a^{(ba)^i c a} = a^{c a (ab)^i} &= a^{(ca\bar{c}a) ac (ab)^i} = b^{(ab)^{k-1} ac (ab)^i} = b^{(ab)^{k-1} (ba)^{i} ac} \\
&= b^{(ab)^{k-1-i} ac} = y_{2(k-1)-2i + 1,1} = y_{2k-2i-1,1} \\
x_{2i,1}^b = a^{(ba)^i c b} = a^{c b (ab)^i} &= a^{(ca\bar{c}a) ac (ab)^{i+1}} = b^{(ab)^{k-1} ac (ab)^{i+1}} = b^{(ab)^{k-1} (ba)^{i+1} ac} \\
&= b^{(ab)^{k-2-i} ac} = y_{(2k-4-2i) + 1,1} = y_{2k-2i-3,1}.
\end{align*}
For parts (5)-(8), recall that $x_{2i+1} = x_{-1-(2i+1)} = x_{2(-i-1)}$, and apply parts (1)-(4), replacing $i$ with $-i-1$ in each formula.  Parts (9)-(16) simply reverse the first eight formulas.  For example, reversing formula (1), we find $y_{2i,0}^a = y_{2k+2(i-k+1)-2,0}^a = x_{2i-2k+2,0} = x_{-1-(2i-2k+2),0} = x_{2k-2i-3,0}$.
\end{proof}

This traces out all the remaining edges. In Figure \ref{F:k=3}, we divide this component into two parts for clarity, one showing the edges labeled $a$ and $b$ at $x_{i,0}$, and the other showing the edges at $x_{i,1}$.

Now we need to consider the secondary relations, and confirm that there is no additional collapsing in the graph. We have five secondary relations:
\begin{alignat*}{2}
W_1&: c^{ba} = c &&\implies x^{abcba\bar{c}} = x \\
W_2&: a^{ca\bar{c}a} = a^{\bar{c}ac} &&\implies x^{aca\bar{c}aca\bar{c}a} = x^{\bar{c}aca\bar{c}ac} \\
W_3&: a^{\bar{c}ac} = b^{(ab)^{k-1}} &&\implies x^{\bar{c}aca\bar{c}ac} = x^{(ba)^{2k-2}b} \\
W_4&: a^{(ba)^{3(2k-1)}} = a &&\implies x^{(ab)^{6(2k-1)}} = x \\
W_5&: b^{(ab)^{3(2k-1)}} = b &&\implies x^{(ba)^{6(2k-1)}} = x 
\end{alignat*}
Note that secondary relations $W_4$ and $W_5$ are equivalent, so we only need to consider one of them. We will verify that each secondary relation holds at $x = x_{2i}$ using Lemma \ref{L:struts}; the proofs for other vertices are very similar.
\begin{align*}
W_1:\ &x_{2i}^{abcba\bar{c}} = x_{2(i-1)}^{cba\bar{c}} = x_{2(i-1),1}^{ba\bar{c}} = y_{2k-3-2(i-1),1}^{a\bar{c}} = y_{2k-2i-1,1}^{a\bar{c}} = x_{2i,1}^{\bar{c}} = x_{2i}. \\
W_2:\ &x_{2i}^{aca\bar{c}aca\bar{c}a} = x_{2i-1}^{ca\bar{c}aca\bar{c}a} = x_{2(i-1)+1,1}^{a\bar{c}aca\bar{c}a} = y_{2k+2(i-1)+1,1}^{\bar{c}aca\bar{c}a} = y_{2k+2i-1}^{aca\bar{c}a} = y_{2k+2i-2}^{ca\bar{c}a} \\
&\ \ = y_{2k+2i-2,1}^{a\bar{c}a} = x_{2k+2(k+i-1),1}^{\bar{c}a} = x_{4k+2i-2}^a = x_{4k+2i-3}\\
&x_{2i}^{\bar{c}aca\bar{c}ac} = x_{2i,0}^{aca\bar{c}ac} = y_{2k+2i-2,0}^{ca\bar{c}ac} = y_{2k+2i-2}^{a\bar{c}ac} = y_{2k+2i-1}^{\bar{c}ac} = y_{2(k+i-1)+1,0}^{ac} \\
&\ \ = x_{2k+2(k+i-1)-1,0}^c = x_{4k+2i-3} = x_{2i}^{aca\bar{c}aca\bar{c}a} \\
W_3:\ &x_{2i}^{(ba)^{2k-2}b} = x_{2i+4k-4}^b = x_{4k+2i-3} = x_{2i}^{\bar{c}aca\bar{c}ac}\\
W_4:\ &x_{2i}^{(ab)^{6(2k-1)}} = x_{2i + 12(2k-1)} = x_{2i+(12k-6)+(12k-6)} = x_{2i}
\end{align*}

Now we turn to the second component, containing generator $c$. This component is shown again in Figure \ref{F:QNM3c}; it does not depend on $k$. We begin with the generator $c$, denoted $z_0$, and add a loop labeled $c$. Then the primary relation $c^{ba} = c$ implies $c^a = c^b$, so we add a second vertex $z_1 = c^a = c^b$. Next we add vertices $z_2 = z_1^c$ and $z_3 = z_2^c = z_1^{\bar{c}}$ (since $z_1^{c^3} = z_1$).

\begin{figure}[htbp]
$$\includegraphics[width=2.75in]{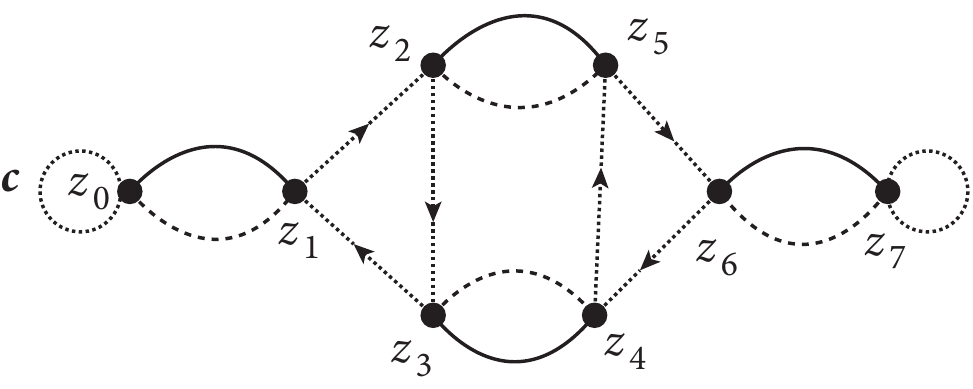}$$
\caption{Component of the Cayley graph for $Q_{(2,3)}(M_3)$ containing generator $c$ (vertex $z_0$). Edges corresponding to the operation of edges $a$, $b$ and $c$ are represented by solid, dashed and dotted lines, respectively.}
\label{F:QNM3c}
\end{figure}

Now we observe, using Lemma \ref{L:r1},
$$z_2^{ab} = c^{acab} = c^{bcbb} = c^{bc} = z_1^c = z_2 \implies z_2^a=z_2^b$$
and
$$z_3^{ab} = c^{a\bar{c}ab} = c^{b\bar{c}bb} = c^{b\bar{c}} = z_1^{\bar{c}} = z_3 \implies z_3^a=z_3^b$$
So we let $z_4 = z_3^a = z_3^b$ and $z_5 = z_2^a = z_2^b$.

For the last couple of steps, we use secondary relation $W_3: x^{\bar{c}aca\bar{c}ac} = x^{(ba)^{2k-2}b}$.  We first apply this at vertex $z_0 = c$.
$$c^{\bar{c}aca\bar{c}ac} = c^{aca\bar{c}ac} = z_1^{ca\bar{c}ac} = z_2^{a\bar{c}ac} = z_5^{\bar{c}ac}$$
and
$$c^{(ba)^{2k-2}b} = c^b = z_1$$
So $z_1 = z_5^{\bar{c}ac}$, which means $z_5^{\bar{c}} = z_1^{\bar{c}a} = z_3^a = z_4$.  Now let $z_6 = z_5^c = z_4^{\bar{c}}$.

Then,
$$z_6^{ab} = c^{ac(aca)b} = c^{ac(bcb)b} = c^{acbc} = z_6.$$
So we let $z_7 = z_6^a = z_6^b$.  The final step is the loop at $z_7$. This results from applying secondary relation $W_3$ at vertex $z_7$.
$$z_7^{(ba)^{2k-1}b} = z_7^b = z_6 \implies z_7^{\bar{c}aca\bar{c}ac} = z_6$$
$$\implies z_7^{\bar{c}} = z_6^{\bar{c}aca\bar{c}a} = z_5^{aca\bar{c}a} = z_2^{ca\bar{c}a} = z_3^{a\bar{c}a} = z_4^{\bar{c}a} = z_6^a = z_7.$$
Hence there is a loop labeled $c$ at $z_7$.  This traces out the component shown in Figure \ref{F:QNM3c}.  It only remains to show that the secondary relations do not induce additional collapsing, but this is easily checked for each of the eight vertices.  These completes our construction of the Cayley graph for $Q_{(2,3)}(M_k)$.

Reviewing our results, we see that the component containing generators $a$ and $b$ has vertices $x_i, x_{i,0}, x_{i,1}, y_i, y_{i,0}$ and $y_{i,1}$ for $0 \leq i \leq 6k-4$, so the component has a total of $6(6k-3) = 18(2k-1)$ vertices.  The second component has 8 vertices, for any value of $k$, for a total of $18(2k-1)+8$ vertices.  This completes the proof of Theorem \ref{T:QNMk}.

\section{Open questions and future work}

We have proved one direction of our Main Conjecture: \medskip

\noindent {\bf Main Conjecture.} {\em A link $L$ with $k$ components has a finite $(n_1,\dots, n_k)$-quandle if and only if there is a spherical orbifold with underlying space $\mathbb{S}^3$ whose singular locus is the link $L$, with component $i$ labeled $n_i$.} \medskip

The remaining, and harder, problem is to prove the other direction -- namely, that the links studied in this paper are the {\em only} ones with finite $N$-quandles. The corresponding proof for $n$-quandles \cite{HS2} uses a relationship between the $n$-quandle and the fundamental group of the $n$-fold branched cover over the link found by Winker \cite{WI}. It is not clear how to define a branched cover with different branching orders over different components of the link, so this argument seems difficult to extend.

Another interesting question is whether this work extends from links to spatial graphs. Quandles can also be defined for spatial graphs, and Dunbar's classification of geometric 3-orbifolds includes many whose singular sets are graphs rather than links.  In these cases, there are often different labels on different edges, so it is natural to consider $N$-quandles of the spatial graphs in Dunbar's classification, and ask whether they are finite. We are currently engaged in investigating the $N$-quandles for these graphs.

\section*{References}

\bibliography{Nquandle.bib}

\end{document}